\documentclass[11pt]{article}

\usepackage[utf8]{inputenc}
\usepackage{geometry}
\usepackage{graphicx}

\usepackage{booktabs}
\usepackage{array}
\usepackage{multirow}
\usepackage{paralist}
\usepackage{verbatim} 
\usepackage{subfig}
\usepackage[usenames,svgnames,dvipsnames]{xcolor}
\usepackage{amsmath}
\usepackage{amssymb}
\usepackage{amsthm}
\usepackage{mathrsfs}
\usepackage{tikz}
\usepackage{tikz-cd}
\usetikzlibrary{calc,fadings,decorations.pathreplacing}
\usepackage{pgfplots}
\pgfplotsset{small,compat=newest}
\usepgfplotslibrary{dateplot}
\usepackage{pgf}
\usetikzlibrary{arrows}
\usetikzlibrary{matrix}
\usetikzlibrary{pgfplots.groupplots}
\usepackage[round]{natbib}
\usepackage{float}
\usepackage{subfig}
\usepackage[pdftex,plainpages=false,colorlinks=true,citecolor=DarkGreen,linkcolor=RoyalBlue,urlcolor=black,filecolor=black,bookmarksopen=true]{hyperref}
\usepackage{setspace}
\usepackage{enumitem}
\usepackage{amsbsy}

\hypersetup{pdftitle={The exponential integrator constraint energy minimizing generalized multiscale method for parabolic problems},
 pdfauthor={Leonardo A. Poveda, Juan Galvis, and Eric T. Chung}}
 
\numberwithin{equation}{section}

\usepackage{fancyhdr}
\usepackage[titles,subfigure]{tocloft} 

\newtheorem{thm}{Theorem}[section]

\newtheorem{lem}[thm]{Lemma}

\newtheorem{asp}[thm]{Assumption}

\theoremstyle{definition}
\newtheorem{dfn}[thm]{Definition}
\newtheorem{example}[thm]{Example}

\theoremstyle{remark}

\numberwithin{equation}{section}

\def\ie{\emph{i.e.}}

\def\fall{\mbox{for all }}
\def\feac{\mbox{for each }}

\def\dive{\mathrm{div}}
\def\spn{\mathrm{span}}
\def\argmin{\mathrm{argmin}}

\def\ms{\mathrm{ms}}

\def\aux{\mathrm{aux}}

\def\E{\mathrm{E}}
\def\RK{\mathrm{RK}}
\def\Or{{\cal O}}
\def\Om{\Omega}
\def\pOm{\partial\Om}

\def\oi{\omega_{i}}

\def\Ki{K_{i}}
\def\Kim{K_{i,m}}
\def\lij{\lambda_{j}^{(i)}}
\def\pij{\varphi_{j}^{(i)}}
\def\psijms{\psi^{(i)}_{j,\ms}}
\def\k{\kappa}
\def\wk{\widetilde{\k}}

\def\gd{\nabla}
\def\p{\partial}
\def\pt{\p_{t}}
\def\ptt{\p_{tt}}
\def\R{\mathbb{R}}

\def\Rd{\R^{d}}

\def\L{\mathrm{L}}
\def\H{\mathrm{H}}

\def\V{\mathrm{V}}
\def\Vz{\V_{0}}
\def\nV{\widetilde{\V}}

\def\Ho{\H^{1}(\Om)}

\def\HoKi{\H^{1}(\Ki)}
\def\HoKim{\H^{1}(\Kim)}
\def\HozKim{\H^{1}_{0}(\Kim)}
\def\Vh{\V_{h}}

\def\Th{{\cal T}^{h}}
\def\TH{{\cal T}^{H}}
\def\Vax{\V_{\aux}}
\def\Viax{\V_{\aux}^{(i)}}
\def\Vms{\V_{\ms}}

\def\x{\mathrm{x}}

\def\Nc{N_{\mathrm{c}}}

\def\Nm{{\cal N}^{v}}
\def\Nf{N_{f}}
\def\Nt{N_{t}}
\def\dt{\delta^{n}}

\def\uh{u_{h}}
\def\vh{v_{h}}
\def\uhn{u_{h}^{n}}
\def\uhnm{u_{h}^{n-1}}
\def\ums{u_{\ms}}
\def\umsn{\ums^{n}}
\def\umsnm{\ums^{n-1}}
\def\pu{\hat{u}}
\def\pw{\hat{w}}

\def\Am{A}
\def\Mm{M}
\def\Amz{\Am_{0}}
\def\Mmz{\Mm_{0}}
\def\Rmz{R_{0}}
\def\Nm{N}

\def\Lms{\Nm_{\ms}}
\def\Im{I}
\def\Lh{L_{h}}
\def\Ph{P_{h}}
\def\Lms{L_{\ms}}
\def\Phf{\Ph f}
\def\Pms{P_{\ms}}
\def\Pmsf{\Pms f}
\def\Bv{B}

\def\Fv{F}
\def\Fvz{\Fv_{0}}
\def\tu{\hat{u}_{h}}
\def\uhn{\uh^{n}}
\def\uhnm{\uh^{n-1}}

\def\uhnRKo{u^{n}_{\E\RK1}}

\def\uhnRKt{u^{n}_{\E\RK22}}

\def\uhnmsro{u^{n}_{\ms,\E\RK1}}
\def\uhnmsrt{u^{n}_{\ms,\E\RK22}}
\def\uhz{\uh^{0}}
\def\Rv{r}

\def\uhms{u_{\ms}}
\def\uhnms{\uhms^{n}}
\def\uhnmms{\uhms^{n-1}}
\def\en{\varepsilon^{n}}
\def\enm{\varepsilon^{n-1}}

\title{A second-order exponential integration constraint energy minimizing generalized multiscale method for parabolic problems}

\author{Leonardo A. Poveda\thanks{Department of Mathematics, The Chinese University of Hong Kong, {\tt lpoveda@math.cuhk.de.hk}} \and Juan Galvis\thanks{Departamento de Matem\'aticas, Universidad Nacional de Colombia,  {\tt jcgalvisa@unal.edu.co} } \and Eric T. Chung\thanks{Department of Mathematics, The Chinese University of Hong Kong, {\tt eric.t.chung@cuhk.edu.hk}}}
\date{}

\begin{document}
\maketitle
\begin{abstract}
This paper investigates an efficient exponential integrator generalized multiscale finite element method for solving a class of time-evolving partial differential equations in bounded domains. The proposed method first performs the spatial discretization of the model problem using constraint energy minimizing generalized multiscale finite element method (CEM-GMsFEM). This approach consists of two stages. First, the auxiliary space is constructed by solving local spectral problems, where the basis functions corresponding to small eigenvalues are captured. The multiscale basis functions are obtained in the second stage using the auxiliary space by solving local energy minimization problems over the oversampling domains.  The basis functions have exponential decay outside the corresponding local oversampling regions. We shall consider the first and second-order explicit exponential Runge-Kutta approach for temporal discretization and to build a fully discrete numerical solution. The exponential integration strategy for the time variable allows us to take full advantage of the CEM-GMsFEM as it enables larger time steps due to its stability properties. We derive the error estimates in the energy norm under the regularity assumption. Finally, we will provide some numerical experiments to sustain the efficiency of the proposed method.
\end{abstract}

\section{Introduction}

In recent decades, the scientific community has devoted considerable effort to developing efficient numerical methods to approximate nonlinear and semilinear parabolic partial differential equations in high-contrast media, which arise from various practical problems. In general, it is hard to find the exact solutions of this kind of model; numerical approaches are currently the essential tools to approximate these solutions. As well-known, these approximations are made at two stages. First, high-contrast ratios require fine-scale meshes in spatial discretization, drastically increasing the degrees of freedom and producing inadequate and inefficient computational costs. On the other hand, high-contrast ratios frequently deteriorate the convergence of the fine-scale approximation, which becomes a common challenge. There have been many existing approaches to handle high-contrast problems, which are usually referred to as numerical upscaling methods, multiscale finite element methods, variational multiscale methods, heterogeneous multiscale methods, mortar multiscale methods, localized orthogonal decomposition methods (LOD),  generalized multiscale finite element methods (GMsFEM), and so on in the literature \cite{Hou1997multiscale, Hughes1998variational,Efendiev2013generalized,Arbogast2013multiscale,chung2023multiscale}. The proposed method in this work belongs to the class of generalized multiscale finite element methods, where the objective is to seek multiscale basis functions to represent the local heterogeneities by solving local constraint-minimizing problems. This approach was initially proposed by \cite{chung2018constraint} and has widely been applied to many applications such as in \cite{li2019constraint,fu2020constraint,chung2020computational,wang2021constraint,poveda2023convergence}. In particular, CEM-GMsFEM can be divided into two steps. First, this method constructs auxiliary basis functions via local spectral problems. Then, constraint energy minimization problems are solved to obtain the required multiscale basis functions. These basis functions are shown to decay exponentially from the target coarse block and can be computed locally.

In the second stage, explicit, semi-implicit, and fully-implicit schemes are frequently used for temporal discretization. These latter are unconditionally stable in contrast with explicit methods, which could be easier to compute but require certain constraints in the time step size. Nevertheless, in implicit schemes, we need to solve non-linear equations at each time step using some iterative method, which can be a bottleneck
in computations. In recent years, alternative techniques have emerged to solve the corresponding nonlinearity. Among the most outstanding, the so-called Exponential Integrators are a good candidate for this purpose. These schemes are robust time-stepping methods that do not need the solution of large linear systems \cite{hochbruck1998exponential,cox2002exponential,hochbruck2010exponential}. Instead, they solve explicitly the problem using the matrix exponential computation for each time step. In literature, there exist several types of methods such as exponential Runge-Kutta methods \citep{hochbruck2005exponential,hochbruck2005explicit}, exponential Rosenbrock methods \citep{hochbruck2009exponential,caliari2009implementation}, exponential multistep method \citep{hochbruck2011exponential}, exponential splitting methods \citep{hansen2009exponential} and Lawson methods \citep{lawson1967generalized}. 

This paper is mainly motivated by \cite{contreras2023exponential,huang2023efficient}. We design and analyze an exponential integration CEM-GMsFEM for semilinear parabolic problems in high-contrast multiscale media. We present a rigorous convergence analysis, which has been lacking so far. The proof is provided under some assumption on the nonlinear reaction term and exact solution.

The remainder of this paper is as follows. Section~\ref{sec:prob-setup} describes the problem and its spatial discretization. The construction of CEM-GMsFEM basis functions using constraint energy minimization is given in Section~\ref{sec:basis}. The multiscale basis functions are constructed by solving a class of local spectral problems and constrained minimization problems. In Section~\ref{sec:expo-inte}, we present the explicit exponential Runge-Kutta method. Under appropriate assumptions of the exact solution and nonlinear reaction term, we show the fully discrete error analysis in Appendix~\ref{sec:convergence}. Numerical experiments are presented in Section~\ref{sec:numerical}. Finally,  conclusions and final remarks are drawn in Section~\ref{sec:conclusion}.

\section{Problem setup}
\label{sec:prob-setup}
In this section, we study the numerical solution of the semilinear parabolic problem by taking the following form:
\begin{equation}
\label{eq:strong-prob}
\left\{\begin{split}
\pt u -\dive(\k(\x)\gd u) &= f(u),\quad\mbox{in }\Om\times[0,T],\\
u(\x,t) &= 0,\quad\mbox{on }\pOm\times[0,T],\\
u(\x,t=0) &= \hat{u},\quad\mbox{on }\Om,
\end{split}\right.
\end{equation}
where $\Om$ is an open domain in $\Rd\,(d=2,3)$ with a boundary defined by $\pOm$, $u(t,\x)$ is the unknown function, $\k$ denotes the high-contrast multiscale field, such that $\k_{0}\leq \k\leq \k_{1}$, where $0<\k_{0}<\k_{1}<\infty$ and $f(u)$ is the nonlinear reaction term of the underlying system and explicitly independent of time. For simplicity of presentation, we present the parabolic problems subject to homogeneous Dirichlet boundary conditions. However, we cite to  \cite{ye2023constraint} for a detailed analysis of CEM-GMsFEM for high-contrast elliptic problems with inhomogeneous boundary conditions.

We briefly present the notation of the function spaces and norms to be used throughout this paper. We denote by $\H^{m}(D)$, with $m\geq 0$ the Sobolev spaces on subdomain $D\subset\Om$ equipped with the norm $\|\cdot\|_{m,D}$. If $D=\Om$, we will omit the index $D$. We denote by $\|\cdot\|_{0,D}$ the norm associated with the inner product $(\cdot,\cdot)$ in the space $\L^{2}(D)$. We further set $\H^{1}_{0}(D)$ the subspace of $\H^{1}(D)$. We will use the notation $x\preceq y$ to show us that there exists a positive constant $C$ independent of the grid size, such as $x\leq Cy$.

\subsection{A semi-discretization by finite element grid approximation}
In this subsection, we introduce the notions of fine and coarse grids to discretize the problem \eqref{eq:strong-prob}. Let $\TH$ be a usual conforming partition of the computational domain $\Om$ into coarse block $K\in\TH$ with diameter $H$. Then, we denote this partition as the coarse grid and assume each coarse element is partitioned into a connected union of fine-grid blocks. In this case, the fine grid partition will be denoted by $\Th$ and is, by definition, a refinement of the coarse grid $\TH$, such that $h\ll H$. We shall denote $\{\x_{i}\}_{i=1}^{\Nc}$ as the vertices of the coarse grid $\TH$, where $\Nc$ denotes the number of coarse nodes. We define the neighborhood of the node $\x_{i}$ by 
\[
\oi=\bigcup\{K_{j}\in\TH:\x_{i}\in\overline{K}_{j}\}.
\]
In addition, for CEM-GMsFEM considered in this paper, we have that given a coarse block $\Ki$, we represent the oversampling region $\Kim\subset\Om$ obtained by enlarging $\Ki$ with $m\geq1$ coarse grid layers, see Fig.~\ref{fig:grid}.
\begin{figure}[h!]
\centering
\begin{tikzpicture}[scale=4.7]
\draw[step=0.02cm,color=Gray, line width = 0.01mm,opacity=0.3] (0,0) grid (1,1);
\draw[step=0.1cm,color=Black,line width = 0.2mm] (0,0) grid (1,1);
\draw [draw=MidnightBlue, line width=1.5pt, fill=Blue, opacity=0.50] (0.1,0.1) rectangle (0.6,0.6);
\draw [draw=Sepia, line width=0.02mm, fill=Mulberry, opacity=0.50] (0.04,0.94) rectangle (0.06,0.96);
\draw [draw=MidnightBlue, line width=1.5pt, fill=Green, opacity=0.50] (0.6,0.7) rectangle (0.8,0.9);
\node at (-0.09,0.35) {$K_{i,2}$};
\draw [draw=black, line width=0.3mm, fill=Yellow, opacity=0.6] (0.3,0.3) rectangle (0.4,0.4);
\node at (0.35,0.35) {\tiny{$\Ki$}};
\node at (-0.3,0.95) {\small{\mbox{fine element}}};
\filldraw (0.7,0.8) circle (0.3pt);
\node at (0.73,0.83) {\small{$\x_{i}$}};
\node at (1.10,0.8) {$\oi$};
\node at (0.04,1.05) {\large{$\Om$}};
\coordinate (A) at (-0.1,0.3);
\coordinate (B) at (-0.1,0.41);
\coordinate (C) at (0.1,0.1);
\coordinate (D) at (0.1,0.6);
\coordinate (E) at (0.05,0.95);
\coordinate (F) at (-0.1,0.95);

\coordinate (G) at (1.09,0.75);
\coordinate (H) at (1.09,0.85);
\coordinate (I) at (0.8,0.7);
\coordinate (J) at (0.8,0.9);

\draw [<-] (A) -- (C);
\draw [<-] (B) -- (D);
\draw [<-] (E) -- (F);
\draw [<-] (G) -- (I);
\draw [<-] (H) -- (J);
\end{tikzpicture}
\caption{\label{fig:grid} Illustration of the $2$D multiscale grid with a typical coarse element $\Ki$ and oversampling domain $K_{i,2}$, the fine grid element and neighborhood $\oi$ of the node $\x_{i}$.}
\end{figure}
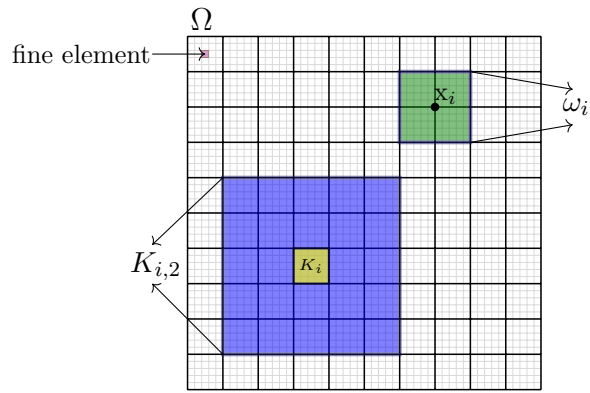

We consider the linear finite element space $\Vh$ associated with the grid $\Th$, where the basis functions in this space are the standard Lagrange basis functions defined as $\{\eta^{i}\}_{i=1}^{\Nf}$, where $\Nf$ denotes the number of interior nodes of $\Th$. Then, the semi-discrete finite element approximation to \eqref{eq:strong-prob} on the fine grid is to find $\uh\in\Vh$ such that
\begin{equation}
\label{eq:weak-prob}
\left\{\begin{split}
(\pt \uh,\vh) +a(\uh,\vh) &= (f(\uh),\vh),\quad\feac \vh\in\Vh,\fall t\in[0,T],\\
(\uh(0),\vh) &= (\hat{u},\vh),\quad\feac \vh\in\Vh.
\end{split}\right.
\end{equation}
Here we use the bilinear forms
\begin{align*}
(u,v)&=\int_{\Om}u(\x)v(\x)d\x,\quad\feac u,v\in\L^{2}(\Om),\\
a(u,v)&=\int_{\Om}\k(\x)\gd u(\x)\cdot\gd v(\x)d\x,\quad\feac u,v\in\H^{1}_{0}(\Om).
\end{align*}
We consider the following representation for the solution of \eqref{eq:weak-prob}:
\begin{equation}
\label{eq:fem-rep}
\uh(\x,t)=\sum_{i=1}^{\Nf}u_{i}(t)\eta_{i}(\x),\quad\fall t\in[0,T].
\end{equation}
We use the expression above into \eqref{eq:fweak-prob} and take $v=\eta_{j}$ for 
$j=1,\dots,\Nf$, then
\begin{equation}
\label{eq:xweak-prob}
\left\{\begin{split}
\sum_{i=1}^{\Nf}\frac{d}{dt}u_{i}(t)(\eta_{i},\eta_{j})+\sum_{i=1}^{\Nf}u_{i}(t)a(\eta_{i},\eta_{j})&=\left(f\left(\sum_{i=1}^{\Nf}u_{i}(t)\eta_{i}\right),\eta_{j}\right),\quad j=1,\dots,\Nf,\\
\sum_{i=1}^{\Nf}u_{i}(0)(\eta_{i},\eta_{j})&=(\hat{u},\eta_{j}),\quad j=1,\dots,\Nf.
\end{split}\right.
\end{equation}
We recast \eqref{eq:xweak-prob} in a continuous-time matrix formulation
\begin{equation}
\label{eq:mat-prob}
\left\{\begin{split}
\Mm\frac{d}{dt}\uh(t)+\Am\uh(t)&=\Fv(\uh(t)),\\
\Mm\uh(0) &= \hat{u},
\end{split}\right.
\end{equation}
where matrices $\Mm,\Am\in\R^{\Nf\times\Nf}$, and term $\Fv(\uh)\in\R^{\Nf}$, with
\begin{equation}
\begin{split}
\label{eq:def-mat}
[\Mm]_{ij} & = \int_{\Om}\eta_{i}(\x)\eta_{j}(\x)d\x,\quad [\Am]_{ij} = \int_{\Om}\k(\x)\gd\eta_{i}(\x)\cdot\gd\eta_{j}(\x)d\x,\\
\Fv_{i}(\uh) & =\int_{\Om}f(\uh)\eta_{i}(\x)d\x,\quad \uh = [u_{1}(t),\dots,u_{\Nf}(t)].
\end{split}
\end{equation}

\section{Construction of CEM-GMsFEM basis functions}
\label{sec:basis}
In this section, we shall describe the construction of CEM-GMsFEM basis functions using the framework from \cite{chung2018constraint}. This procedure can be divided into two stages. The first stage involves constructing the auxiliary spaces by solving a local spectral problem in each coarse element $K$. The second stage provides the multiscale basis functions by solving local constraint energy minimization problems in oversampling regions, see \cite{chung2018constraint}.

\subsection{Auxiliary basis function}

We present the construction of the auxiliary multiscale basis functions by solving the local eigenvalue problem for each coarse element $\Ki$. We consider $\HoKi:=\Ho\big|_{\Ki}$ the restriction of the space $\Ho$ to the coarse element $\Ki$. We solve the following local eigenvalue problem: find $\{\lij,\pij\}$ such that
\begin{equation}
\label{eq:eigen-prob}
a_{i}(\pij,w)=\lij s_{i}(\pij,w),\quad\feac w\in\HoKi,
\end{equation}
where
\[
a_{i}(v,w):=\int_{\Ki}\k\gd v(\x)\cdot\gd w(\x)d\x,\quad s_{i}(v,w):=\int_{\Ki}\wk v(\x)w(\x)d\x.
\]
Here, $\wk=\k\sum_{i=1}^{\Nc}|\gd\chi_{i}|^{2}$, where $\Nc$ is the total number of neighborhoods and $\{\chi_{i}\}$ is a set of partition of unity functions for the coarse grid $\TH$. The problem defined above is solved on the fine grid in the actual computation. We assume that the eigenfunctions satisfy the normalized condition $s_{i}(\pij,\pij)=1$. We shall use $L_{i}$ eigenvectors corresponding to the first $\lij$ eigenvalues that are arranged in increasing order and to construct the local auxiliary multiscale space $\Viax:=\{\pij:1\leq j\leq L_{i}\}$. We can define the global auxiliary multiscale space as $\Vax:=\bigoplus_{i=1}^{\Nc}\Viax$.

For the local auxiliary space $\Viax$, the bilinear form $s_{i}$ given above defines an inner product with norm $\|v\|_{s(\Ki)}=s_{i}(v,v)^{1/2}$. Then, we can define the inner product and norm for the global auxiliary multiscale space $\Vax$, which are defined by
\[
s(v,w)=\sum_{i=1}^{\Nc}s_{i}(v,w),\quad\|v\|_{s}:=s(v,v)^{1/2},\quad\feac v,w\in\Vax.
\]
To construct the CEM-GMsFEM basis functions, we use the following definition.

\begin{dfn}[\citealp{chung2018constraint}]
Given a function $\pij\in\Vax$, if a function $\psi\in\V$ satisfies 
\[
s(\psi,\pij):=1,\quad s(\psi,\varphi_{j'}^{(i')})=0,\quad\mbox{if }j'\neq j\mbox{ or }i'\neq i,
\]
then, we say that is $\pij$-orthogonal where $s(v,w)=\sum_{i=1}^{\Nc}s_{i}(v,w)$.
\end{dfn}
Now, we define $\pi:\V\to\Vax$ as the projection of the inner product $s(v,w)$. More precisely, $\pi$ is defined by
\[
\pi(v):=\sum_{i=1}^{\Nc}\pi_{i}(v)=\sum_{i=1}^{\Nc}\sum_{j=1}^{L_{i}}s_{i}(v,\pij)\pij,\quad\feac v\in\V,
\]
where $\pi_{i}:L^{2}(\Ki)\to\Viax$ denotes the projection with respect to inner product $s_{i}(\cdot,\cdot)$. The null space of the operator $\pi$ is defined by $\nV=\{v\in\V:\pi(v)=0\}$.
Now, we will construct the multiscale basis functions. Given a coarse block $\Ki$, we denote the oversampling region $\Kim\subset\Om$ obtained by enlarging $\Ki$ with an arbitrary number of coarse grid layers $m\geq1$, see Figure \ref{fig:grid}. Then, we define the multiscale basis functions by 
\begin{equation}
\label{eq:main-argmin}
\psijms=\argmin\{a(\psi,\psi):\psi\in\HozKim,\,\psi\mbox{ is $\pij$-orthogonal}\},
\end{equation}
where $\HoKim$ is the restriction of $\Ho$ in $\Kim$ and $\HozKim$ is the subspace of $\HoKim$ with zero trace on $\p\Kim$. The multiscale finite element space $\Vms$ is defined by
\[
\Vms=\spn\{\psijms:1\leq j\leq L_{i}, 1\leq i\leq \Nc\}.
\]
By introducing the Lagrange multiplier, the problem \eqref{eq:main-argmin} is equivalent to the explicit form: find $\psijms\in\HozKim$, $\xi\in\Viax(\Ki)$ such that
\[
\begin{cases}
a(\psijms,\mu)+s(\mu,\upsilon)&=0,\quad\fall \mu\in \HozKim,\\
s(\psijms-\pij,\nu)&=0,\quad\fall \nu\in\Viax(\Kim),
\end{cases}
\]
where $\Viax(\Kim)$ is the union of all local auxiliary spaces for $\Ki\subset\Kim$. Thus, the semi-discrete multiscale approximation reads as follows: find $\ums\in\Vms$ such that
\begin{equation}
\label{eq:sms-prob}
(\pt \ums,v) +a(\ums,v) = (f(\ums),v),\quad\feac v\in\Vms.
\end{equation}
Each multiscale basis function $\psijms$ is eventually represented on the fine grid. Therefore, each $\psijms$ can be represented by a vector. Using \eqref{eq:main-argmin}, we shall construct the coarse-scale matrix of local basis functions as:
\[
\Rmz = [\psi_{1,\ms}^{(i)},\dots,\psi_{L_{i},\ms}^{(i)}],
\]
which maps quantities from the multiscale space $\Vms$ to the fine-scale space $\Vh$. Similarly to \eqref{eq:mat-prob}, we can define the matrix coarse-scale nonlinear system as
\begin{equation}
\label{eq:mat-cprob}
\Mmz\frac{d}{dt}\uhms^{H}+\Amz\uhms^{H}=\Fvz(\uhms^{H}),
\end{equation}
where $\uhms^{H}$ is the coarse-scale approximation, and the coarse-scale stiffness and mass matrices are given by
\begin{equation}
\label{eq:coarse-mat}
\Amz=\Rmz^{T}\Am\Rmz,\quad\Mmz=\Rmz^{T}\Mm\Rmz,
\end{equation}
respectively, and $T$ represents the transpose operator. We also have the coarse-scale vector as $\Fvz=\Rmz^{T}\Fv$.

\section{Temporal discretization}
\label{sec:expo-inte}
This section considers the fully-discrete scheme for the discrete formulation \eqref{eq:weak-prob}. Let $0=t_{0}<t_{1}<\cdots<t_{\Nt-1}<t_{\Nt}=T$ be a partition of the interval $[0,T]$, with time-step size given by $\dt=t_{n}-t_{n-1}>0$,  for $n=1,\dots,\Nt$, where $\Nt$ is an integer. 


\subsection{Classical implicit time integration schemes}
\label{ssec:fd-scheme}
To solve the ODE matrix systems \eqref{eq:mat-prob} and \eqref{eq:mat-cprob} in an interval $[0, T]$ and comparing the proposed method, we introduce the  implicit time integration scheme, which leads to finding $\uhn\in\Vh$ such that
\begin{equation}
\label{eq:fweak-prob}
(\uhn-\uhnm,v)+\dt a(\uhn,v)=\dt(f(\uhn),v),\quad\feac v\in\Vh,
\end{equation}
with a small enough time step size $\dt$. By using \eqref{eq:fem-rep} into \eqref{eq:fweak-prob} and take $v=\eta_{j}$ for $j=1,\dots,\Nf$, we then have the matrix problem
\begin{equation}
\label{eq:full-mat-prob}
\left\{\begin{split}
\Mm(\uhn-\uhnm)&=\dt\theta\left(\Fv(\uhn)-\Am\uhn\right)+\dt(1-\theta)\left(\Fv(\uhnm)-\Am\uhnm\right),\\
\Mm\uhz &= \tu,
\end{split}\right.
\end{equation}
where $\tu=\int_{\Om}\hat{u}\eta_{j}dx$ and matrices $\Mm,\Am$ and vector $\Fv(\uhn)$ are defined as in \eqref{eq:def-mat}. If $\theta=1$, we obtain the backward Euler scheme. On the other hand, by taking $\theta=\tfrac{1}{2}$, we obtain the Crank-Nicolson scheme.

 Analogously, we can define a full-discrete multiscale formulation: find $\umsn\in\Vms$ such that
\begin{equation}
\label{eq:fms-prob}
(\umsn-\umsnm,v) +\dt a(\umsn,v) = \dt(f(\umsn),v),\quad\feac v\in\Vms.
\end{equation}
Therefore, denoting the fine-scale residue by $\Rv=\Mm\uhnm+\dt\Fv(\uhn)$, we obtain that
\[
\uhnms = \Rmz(\Mmz + \dt\theta\Amz)^{-1}\Rmz^{T}\left((\Mm-\dt(1-\theta)\Am)\uhnmms+\dt((1-\theta)\Fv(\uhnmms)-\Fv(\uhnms))\right). 
\]

\subsection{Explicit exponential time integration}

In this subsection, we recall exponential integration techniques relevant to the proposed constraint energy minimizing generalized multiscale finite element approach. Firstly, We define the discrete solution $\uh(t)$ evolving in time. Since that $a(\cdot,\cdot)$ is a bounded bilinear functional over $\Vh\times\Vh$, by invoking Riesz's representation Theorem, we have that there exists a bounded linear operator $\Lh:\Vh\to\Vh$ such that
\[
a(\uh,\vh)=\langle\Lh\uh,\vh\rangle,\quad\feac\vh\in\Vh.
\]
Then, we can be rewritten \eqref{eq:weak-prob} as
\begin{equation}
\label{eq:Lh-fem}
\begin{cases}
\pt\uh+\Lh\uh &=\Phf(\uh),\quad \feac\x\in\Om,\quad 0\leq t\leq T,\\
\uh(0)& = \Ph\hat{u},\quad \x\in\Om,
\end{cases}
\end{equation}
where $\Ph:\L^{2}\to\Vh$ is the $\L^{2}$-orthogonal projection operator.  Analogous to the matrix problem \eqref{eq:mat-prob},  we assume that $\uh(t_{n-1})$ is given for a current time $t_{n-1}$, and we pretend to compute $\uh(t_{n})$, with $t_{n}=t_{n-1}+\dt$. Thus, we introduce the integrating factor problem given by
\begin{equation}
\label{eq:homo-ode}
\frac{d}{dt}\Bv(t)=\Bv(t)\Lh,\quad \Bv(t_{n-1})=\Im,\quad n = 1,\dots,\Nt,
\end{equation}
where $\Im$ is the identity matrix. The problem \eqref{eq:homo-ode} has a unique solution given by	$\Bv(t)=e^{(t-t_{n-1})\Lh}$. Then, using the matrix $\Lh$ and the integration factor above, one can rewrite the problem \eqref{eq:mat-prob} as
\[
\frac{d}{dt}(\Bv(t)\Lh\uh(t))=\Bv(t)\Phf(\uh),
\]
This problem has an exact solution implicitly represented as
\begin{equation}
\label{eq:ex-expint}
\uh(t_{n}) = e^{-\dt\Lh}\uh(t_{n-1})+\int_{0}^{\dt}e^{(s-\dt)\Lh}\Phf(\uh(s+t_{n-1}))ds.
\end{equation}
Notice that $\Bv^{-1}(t):=e^{-(t-t_{n-1})\Nm}$ denotes the inverse of $\Bv(t)$. This last equation is well-known as the variation-of-constants formula \citep{hochbruck2010exponential,hochbruck1998exponential}.

\subsubsection{Numerical time integration}


We now pretend an approximation to the nonlinear term within the integral of \eqref{eq:ex-expint} via an interpolating polynomial on certain quadrature nodes. We denote $\uhn\approx\uh(t_{n})$ the fully discrete numerical solution at the time step $t_{n}$. Then, we shall apply the classic explicit exponential Runge-Kutta scheme \citep{hochbruck2010exponential} and obtain a fully-discrete numerical method for solving the problem \eqref{eq:strong-prob} as follows: for $n=1,\dots,\Nt$,
\begin{equation}
\label{eq:RK-s}
\left\{\begin{split}
U^{n,i} & = e^{-c_{i}\dt\Lh}\uhnm+\dt\sum_{j=1}^{i-1}\alpha_{ij}(-\dt\Lh)\Phf(U^{n,j}),\quad i = 1,\dots, m,\\
\uhn& = e^{-\dt\Lh}\uhnm+\dt\sum_{i=1}^{m}\beta_{i}(-\dt\Lh)\Phf(U^{n,i}),\\
\end{split}\right.
\end{equation}
where $m$ denotes the number of stages for the exponential Runge-Kutta method and $U^{n,i}\approx\uh(t_{n-1}+c_{i}\dt)$.  Here the interpolation nodes $c_{1},\dots,c_{m}$ are $m$ distinct nodes selected in $[0,1]$. The coefficients $\alpha_{ij}(-\dt\Lh)$ and $\beta_{j}(-\dt\Lh)$ are  chosen as linear combinations of the $\phi$-functions $\phi_{k}(-c_{i}\dt\Lh)$ and $\phi_{k}(-\dt\Lh)$, respectively. These functions are given by
\begin{equation}
\label{eq:phi-function}
\phi_{0}(\Nm) = e^{\Nm},\quad\phi_{k}(\Nm) = \int_{0}^{1}e^{(1-\theta)\Nm}\frac{\theta^{k-1}}{(k-1)!}d\theta,\quad k\geq 1.
\end{equation}
These $\phi$-functions satisfy the following recurrence relation (see for instance \citealp{cox2002exponential}),
\[
\phi_{k+1}(\Nm)=\Nm^{-1}\left(\phi_{k}(\Nm)-\frac{1}{k!}\Im\right),\quad\mbox{for }k\leq0,
\]
and we are assuming $\Nm^{-1}$ exists. Thus,
\[
\phi_{k}(-\dt\Lh)=\frac{1}{(\dt)^{k}}\int_{0}^{\dt}e^{-(\dt-\tau)\Lh}\frac{\tau^{k-1}}{(k-1)!}d\tau,\quad \feac k\geq 1.
\]
For reasons of consistency, we will assume throughout the paper that the following assumptions hold \citep{hochbruck2010exponential}, 
\begin{equation}
\label{eq:cons-cond}
\sum_{j=1}^{m}\beta_{j}(-\dt\Lh)=\phi_{1}(-\dt\Lh),\quad\sum_{j=1}^{i-1}\alpha_{ij}(-\dt\Lh)=c_{i}\phi_{1}(-\dt\Lh),\quad 1\leq i\leq m.
\end{equation}
From the latter, we can infer that $c_{1}=0$. Taking $m=1$ in \eqref{eq:RK-s}, we have a first-order scheme given as
\begin{equation}
\label{eq:EIRK1}
\uhnRKo = \uhnm+\dt\phi_{1}(-\dt\Lh)\left[\Phf(\uhnm)-\Lh\uhnm\right],
\end{equation}
which is well-known as the first-order Exponential Euler method. Analogously,
for $m=2$, we have two interpolation nodes $c_{1}=0$ and $c_{2}\in(0,1]$. In particular, we shall use $c_{2}=1$. Thus, we obtain the two-stage second-order exponential Runge-Kutta scheme for  \eqref{eq:RK-s} correspondingly reads
\begin{equation}
\label{eq:EIRK22}
\uhnRKt = \uhnRKo+\dt\phi_{2}(-\dt\Lh)\left[\Phf(\uhnRKo)-\Phf(\uhnm)\right].
\end{equation}
More general (higher-order) exponential time stepping schemes need more complicated order conditions and may be derived using higher-order $\phi$-functions defined above (see \cite{hochbruck2005explicit,hochbruck2005exponential} for more details). Analogously, we propose the approximation
\[
\phi_{k}(-\dt\Lh)\approx \Rmz\phi_{k}(-\dt\Lms)\Rmz^{T},\quad k\geq 1,
\]
where $\Lms$ is the linear operator associated with the bilinear functional $a(\cdot,\cdot)$ in \eqref{eq:sms-prob}. Then, we can define a first-order, fully discrete multiscale exponential integrator formulation
\begin{equation}
\label{eq:ms-RK1}
\uhnmsro = \uhnmms + \dt\Rmz\phi_{1}(-\dt\Lms)\Rmz^{T}\left(\Phf(\uhnmms)-\Lh\uhnmms\right),
\end{equation}
and a second-order formulation
\begin{equation}
\label{eq:ms-RK22}
\uhnmsrt = \uhnmsro + \dt\Rmz\phi_{2}(-\dt\Lms)\Rmz^{T}\left(\Phf(\uhnmsro)-\Phf(\uhnmms)\right),
\end{equation}
for $n=1,\dots,\Nt$.

\section{Numerical experiments}
\label{sec:numerical}
In this section, we present some numerical results by using the exponential integration CEM-GMsFEM to solve the parabolic equation with the multiscale permeability field $\k(\x)$ up to $t=T$ using various time-step sizes $\dt=T/\Nt$ and a fixed number of fine grid nodes. Our numerical experiments evaluate the $\phi$-functions using the Pad\'e approximation implemented EXPINT package in \cite{berland2007expint}. We consider spatial variable $\x=(x_{1},x_{2})$, in $\Om=[0,1]\times[0,1]$ and a $128\times128$ fine grid to compute a reference solution. To show spatial accuracy, we will choose a different coarse grid size to compute the relative error between the fine-scale and first-order CEM-GMsFEM-EIRK1 (Exponential Runge-Kutta) and CEM-GMsFEM-FDBE (Finite Difference Backward Euler) solutions. For temporal accuracy, we consider different time steps in all experiments. We also consider different high-contrast permeability fields $\k_{i}(\x)$, with $i\in\{1,2,3,4\}$ used in the numerical examples shown in Figure~\ref{fig:perm-fields}.
\begin{figure}[t]
\centering
\subfloat[perm1][$\k_{1}$.]{\includegraphics[width=0.25\textwidth]{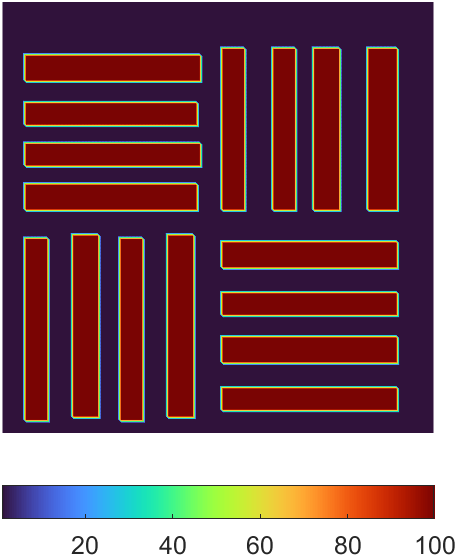}
\label{fig:perm1}}
\subfloat[perm2][$\k_{2}$.]{\includegraphics[width=0.25\textwidth]{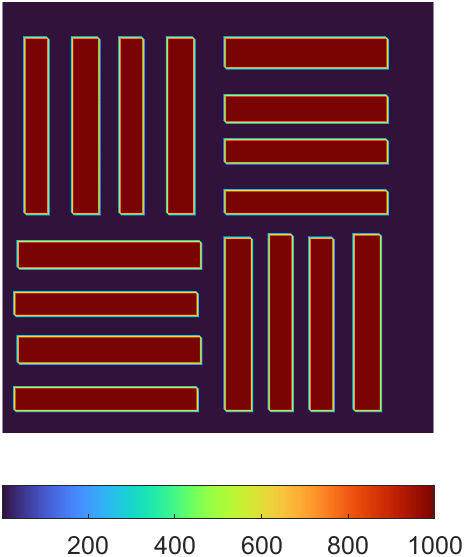}
\label{fig:perm2}}
\subfloat[perm3][$\k_{3}$.]{\includegraphics[width=0.25\textwidth]{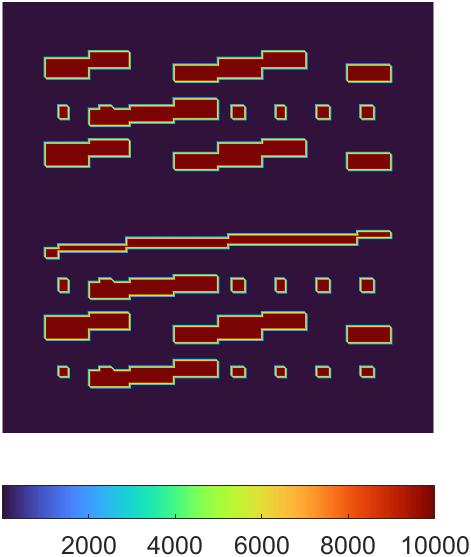}\label{fig:perm3}}
\subfloat[perm4][$\k_{4}$.]{\includegraphics[width=0.25\textwidth]{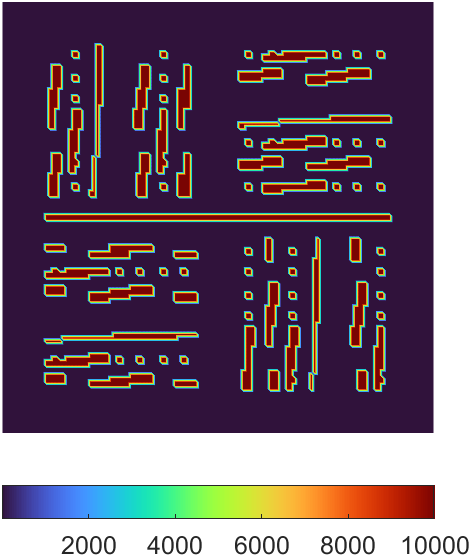}\label{fig:perm4}}
\caption{\label{fig:perm-fields} Permeability fields.}
\end{figure}

In addition, to quantify the accuracy of the multiscale solutions obtained from the proposed method, we define the relative $\L^{2}$, $\H^{1}$ and $\max$ error as follows:
\[
\varepsilon_{0}=\frac{\|\uh^{\Nt}-\ums^{\Nt}\|_{0}}{\|\uh^{\Nt}\|_{0}},\quad \varepsilon_{1}=\frac{\|\uh^{\Nt}-\ums^{\Nt}\|_{a}}{\|\uh^{\Nt}\|_{a}},\quad \varepsilon_{\infty}=\frac{\|\uh^{\Nt}-\ums^{\Nt}\|_{\infty}}{\|\uh^{\Nt}\|_{\infty}},
\]
where $\uh^{\Nt}$ denotes the reference solution at instant $t=T$.

\begin{example}
\label{example1}
In this experiment, we consider the problem:
\begin{equation}
\label{eq:example1}
\left\{\begin{split}
\pt u -\dive(\k_{1}(\x)\gd u)& = 0,\quad\mbox{in }\Om\times[0,T],\\
u(x_{1},x_{2},t)& = 0,\quad\mbox{on }\pOm\times[0,T],\\
u(x_{1},x_{2},t=0)& = x_{1}(1-x_{1})x_{2}(1-x_{2}),\quad\mbox{on }\Om.
\end{split}\right.
\end{equation}

\begin{figure}
\centering
\begin{tikzpicture}
\pgfplotsset{samples=10}
\centering
\begin{groupplot}[
group style = {group size = 2 by 1, 
horizontal sep = 2cm,
vertical sep = 20pt}, 
width  = 6.6cm, 
height = 6cm
]
\nextgroupplot[
xlabel={\small{Number of local basis functions}}, 
xtick={1,...,8},
ylabel={\small{Relative error}},
ymode=log,
legend style = { column sep = 3pt, legend columns = -1, legend to name = grouplegend,}
]
\addplot[Blue,mark=triangle*,mark size=1.2pt] table [x=NB,y=H1] {error/ex1_svd0_nb_200.dat};\addlegendentry{$\varepsilon_{a}$(EIRK1)}
\addplot[BrickRed,mark=diamond*,mark size=1.2pt] table [x=NB,y=L2] {error/ex1_svd0_nb_200.dat};\addlegendentry{$\varepsilon_{0}$(EIRK1)}
\addplot[ForestGreen,mark=square*,mark size=1.2pt] table [x=NB,y=H1] {error/ex1_fd_nb_200.dat};\addlegendentry{$\varepsilon_{a}$(FDBE)}
\addplot[JungleGreen,mark=*,mark size=1.2pt] table [x=NB,y=L2] {error/ex1_fd_nb_200.dat};;\addlegendentry{$\varepsilon_{0}$(FDBE)}
\nextgroupplot[
xlabel={\small{Number of oversampling layers}}, 
xtick={1,...,8},
ylabel={\small{Relative error}},
ymode=log,
]
\addplot[Blue,mark=triangle*,mark size=1.2pt] table [x=OS,y=H1] {error/ex1_svd0_ov_200.dat};
\addplot[BrickRed,mark=diamond*,mark size=1.2pt] table [x=OS,y=L2] {error/ex1_svd0_ov_200.dat};
\addplot[ForestGreen,mark=square*,mark size=1.2pt] table [x=OS,y=H1] {error/ex1_fd_ov_200.dat};
\addplot[JungleGreen,mark=*,mark size=1.2pt] table [x=OS,y=L2] {error/ex1_fd_ov_200.dat};
\end{groupplot}
\node at ($(group c2r1) + (-3.5cm,-4.0cm)$) {\ref{grouplegend}}; 
\end{tikzpicture}
\caption{\label{fig:example1nb} Relative error for the CEM-GMsFEM-EIRK1 and CEM-GMsFEM-FDBE solution with increasing the number of local multiscale basis functions (left) and the number of oversampling layers (right) for problem \eqref{eq:example1} at final time $T=0.2$.}
\end{figure}
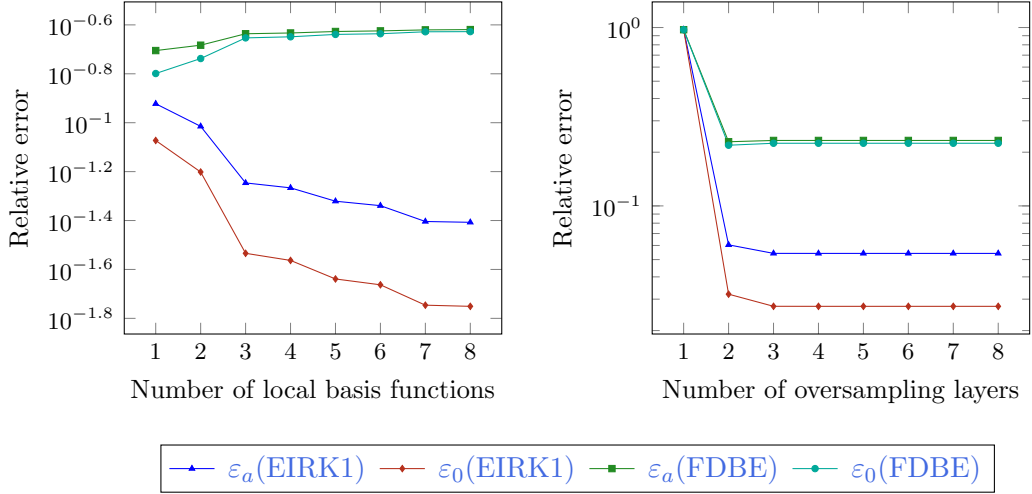

We test with a vanishing reaction term $f:=0$ and use a high-contrast media $\k_{1}$ with the value of $10^{2}$ in the high-contrast channels (see Figure~\ref{fig:perm1}). The final time of this simulation is $T=0.2$, and we consider $200$ time steps for CEM-GMsFEM-EIRK1 and CEM-GMsFEM-FDBE. The reference solution is approximated in the fine-scale grid using the backward Euler scheme with $1000$ time steps. In Figure \ref{fig:example1nb}, we present the relative error estimates between the reference solution and CEM-GMsFEM-EIRK1 and CEM-GMsFEM-FDBE schemes for problem \eqref{eq:example1} with a coarse grid size of $H=\tfrac{1}{8}$. From Figure (left), we observe that the errors decrease as the number of local basis functions increases using CEM-GMsFEM-EIRK1 in contrast to the CEM-GMsFEM-FDBE scheme. Therefore, we obtain a good agreement using only a few local basis functions on each coarse block. From Figure \ref{fig:example1nb} (right), we can observe that the accuracy will improve as the number of oversampling layers increases. Further, when enough oversampling layers are given, the error tends to be constant. 

\begin{table}
\centering
\begin{tabular}{|c|c|c|c|c|c|}
\hline 
Scheme & $H$ & $m$ & $\varepsilon_{a}$ & $\varepsilon_{0}$ & $\varepsilon_{\infty}$\tabularnewline
\hline 
\hline 
\multirow{4}{*}{EIRK1} & $\tfrac{1}{2}$ & 1 & 1.3299E--01 & 8.6967E--02 & 1.2217E--01\tabularnewline
\cline{2-6} \cline{3-6} \cline{4-6} \cline{5-6} \cline{6-6} 
 & $\frac{1}{4}$ & 2 & 1.1094E--01 & 6.9012E--02 & 1.0449E--01 \tabularnewline
\cline{2-6} \cline{3-6} \cline{4-6} \cline{5-6} \cline{6-6} 
 & $\frac{1}{8}$ & 2 & 5.4157E--02 & 2.9707E--02 & 3.9885E--02
\tabularnewline
\cline{2-6} \cline{3-6} \cline{4-6} \cline{5-6} \cline{6-6} 
 & $\frac{1}{16}$ & 3 & 2.5312E--02 & 1.0422E--02 & 1.6499E--02
\tabularnewline
\hline 
\multirow{4}{*}{FDBE} & $\frac{1}{2}$ & 1 & 1.8257E--01 & 1.5295E--01 & 1.5246E--01\tabularnewline
\cline{2-6} \cline{3-6} \cline{4-6} \cline{5-6} \cline{6-6} 
 & $\frac{1}{4}$ & 2 & 1.9529E--01 & 1.7438E--01 & 1.7775E--01\tabularnewline
\cline{2-6} \cline{3-6} \cline{4-6} \cline{5-6} \cline{6-6} 
 & $\frac{1}{8}$ & 2 & 2.2888E--01 & 2.2171E--01 & 2.1371E--01 \tabularnewline
 \cline{2-6} \cline{3-6} \cline{4-6} \cline{5-6} \cline{6-6} 
 & $\frac{1}{16}$ & 3 & 2.5040E--01 & 2.4381E--01 & 2.5981E--01\tabularnewline
\hline 
\end{tabular}
\caption{\label{tab:example1} Spatial convergence rate for problem~\eqref{eq:example1} at the final time $T=0.2$ with varying coarse grid size $H$ and oversampling coarse layers $m$ using a contrast of $10^{2}$.}
\end{table}
Table \ref{tab:example1} shows the convergence behavior concerning the coarse mesh size in $\H^{1}$, $\L^{2}$ and $\max$-norm for the problem~\eqref{eq:example1} at the final time. We only use $4$ basis functions on each coarse block with varying coarse grid sizes $H=\tfrac{1}{2},\tfrac{1}{4},\tfrac{1}{8}$,  and $\tfrac{1}{16}$, associated with an appropriate number of oversampling layers $m$. Observe that for CEM-GMsFEM-EIRK1 with fixed $H=\tfrac{1}{8}$, the relative errors are just about $6.05\%,2.97\%$ and $3.98\%$ in $\H^{1}$, $\L^{2}$ and $\max$-norm respectively. Furthermore, for CEM-GMsFEM-EIRK1 schemes, the errors decay as coarse mesh size. On the other hand, the relative errors for the CEM-GMsFEM-FDBE scheme are just about $22.89\%,22.17\%$, and $21.37\%$ in the respective norms. We also observe that errors are increased even with increasing time steps as we increment coarse mesh size.  Table~\ref{tab:example12} shows the order of temporal accuracy for CEM-GMsFEM-EIRK1 scheme in $\L^{2}$ and $\H^{1}$-norms, where $CR$ is defined as 
\[
CR =\frac{|\ln\varepsilon^{\Nt(i)}_{\star}-\ln\varepsilon^{\Nt(i-1)}_{\star}|}{\ln 2},\quad i \in\{2,...,5\},
\]
where $\star\in\{0,a\}$. We notice that the order of the temporal accuracy in the $\varepsilon_{a}$ is just about $1$, which coincides with the result given in Theorem~\ref{thm:RK1dtH}. In addition, the order of the temporal accuracy in the $\varepsilon_{0}$ is much better than expected from theoretical results given in Theorem~\ref{thm:RK22dt2H2}. Therefore, we have a good performance of the CEM-GMsFEM-EIRK1 scheme. Figure~\ref{fig:example1} depicts the solution profiles at the final instant $T=0.2$ using the two schemes.
\begin{table}
\centering
\begin{tabular}{|c|c|c|c|c|c|}
\hline 
Scheme & $\Nt$ & $\varepsilon_{a}$ & $CR$ & $\varepsilon_{0}$ & $CR$ \tabularnewline
\hline 
\hline 
\multirow{5}{*}{EIRK1} & 8 & 2.1449E+00 & -- & 2.1405E+00 & -- \tabularnewline
\cline{2-6} \cline{3-6} \cline{4-6} \cline{5-6} \cline{6-6}
 & 16 & 7.1026E--01 & 1.5945E+00 & 7.0602E--01 & 1.6001E+00\tabularnewline
\cline{2-6} \cline{3-6} \cline{4-6} \cline{5-6} \cline{6-6} 
 & 32 & 2.6403E--01 & 1.4276E+00 & 2.5743E--01 & 1.4555E+00\tabularnewline
\cline{2-6} \cline{3-6} \cline{4-6} \cline{5-6} \cline{6-6} 
 & 64 & 9.4676E--02 & 1.4796E+00 & 7.9631E--02 & 1.6928E+00\tabularnewline
 \cline{2-6} \cline{3-6} \cline{4-6} \cline{5-6} \cline{6-6} 
 & 128 & 4.8138E--02 & 9.7581E--01 & 4.9546E--03 & 4.0064E+00\tabularnewline
 \hline
\end{tabular}
\caption{\label{tab:example12} Temporal convergence rate for problem~\eqref{eq:example1} with  contrast of $10^{2}$ and final time of $T=0.2$ with varying time steps $\Nt$ and fixed coarse size of $H=\tfrac{1}{8}$ and $4$ basis functions.}
\end{table}
\begin{figure}
\centering
\subfloat[EIexa1][EIRK1.]{\includegraphics[width=0.30\textwidth]{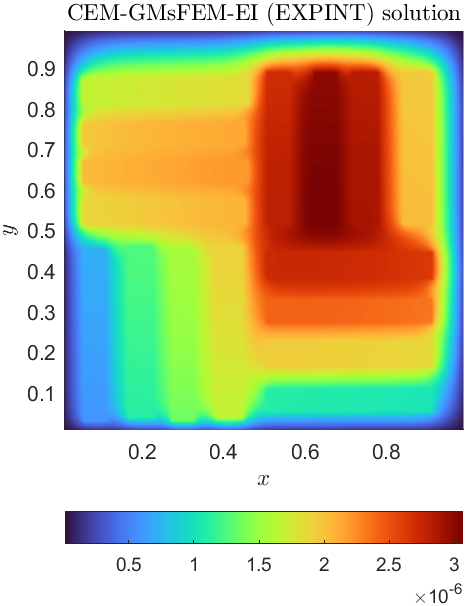}\label{fig:EI1exa1}}~
\subfloat[FDexa1][FDBE.]{\includegraphics[width=0.30\textwidth]{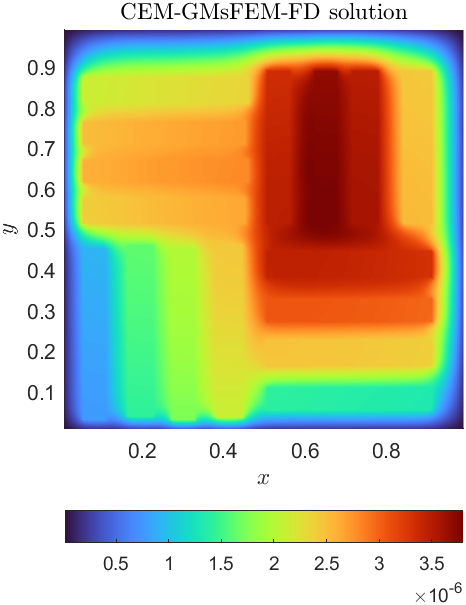}\label{fig:FDexa1}}~
\subfloat[RFexa1][Reference solution.]{\includegraphics[width=0.30\textwidth]{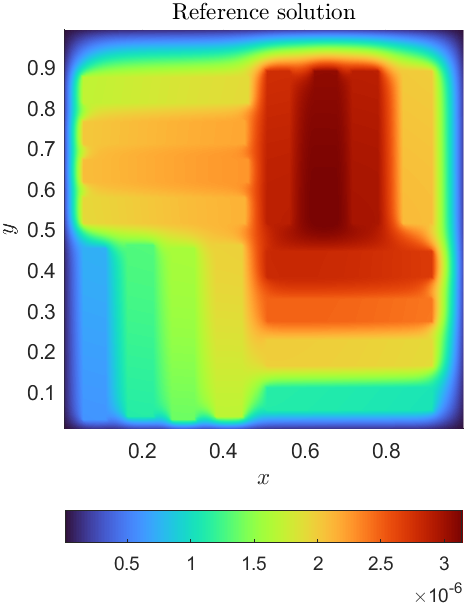}\label{fig:RFexa1}}
\caption{\label{fig:example1} (a) CEM-GMsFEM-EIRK1, (b) CEM-GMsFEM-FDBE, and (c) The reference solution of problem~\eqref{eq:example1} at final time $T=0.2$, using coarse grid size $H=\tfrac{1}{8}$, $4$ local multiscale basis functions and $m=4$ oversampling layers.}
\end{figure}
\end{example}

\begin{example}
In this experiment, we consider the problem:
\begin{equation}
\label{eq:example2}
\left\{\begin{split}
\pt u -\dive(\k_{2}(\x)\gd u)& = u-u^{3},\quad\mbox{in }\Om\times[0,T],\\
u(x_{1},x_{2},t)& = 0,\quad\mbox{on }\pOm\times[0,T],\\
u(x_{1},x_{2},t=0)& = x_{1}(1-x_{1})x_{2}(1-x_{2}),\quad\mbox{on }\Om.
\end{split}\right.
\end{equation}
We use the high-contrast media $\k_{2}$ with the value of $10^{3}$ in the high-contrast channels (see Figure~\ref{fig:perm2}). The final time is $T=0.2$, and we consider $200$ time steps for CEM-GMsFEM-EIRK1 and $300$ time steps for CEM-GMsFEM-FDBE. The reference solution is again approximated in the fine-scale grid using the Backward Euler scheme with $1000$ time steps. 
\begin{figure}
\centering
\begin{tikzpicture}
\pgfplotsset{samples=10}
\centering
\begin{groupplot}[
group style = {group size = 2 by 1, 
horizontal sep = 2cm,
vertical sep = 20pt}, 
width = 6.6cm, 
height = 6cm
]
\nextgroupplot[
xlabel={\small{Number of local basis functions}}, 
xtick={1,...,8},
ylabel={\small{Relative error}},
ymode=log,
legend style = { column sep = 3pt, legend columns = -1, legend to name = grouplegend,}
]
[Blue,mark=triangle*]
\addplot[Blue,mark=triangle*,mark size=1.2pt] table [x=NB,y=H1] {error/ex2_svd0_nb_200.dat};\addlegendentry{$\varepsilon_{a}$(ERK1)}
\addplot[BrickRed,mark=diamond*,mark size=1.2pt] table [x=NB,y=L2] {error/ex2_svd0_nb_200.dat};\addlegendentry{$\varepsilon_{0}$(ERK1)}
\addplot[ForestGreen,mark=square*,mark size=1.2pt] table [x=NB,y=H1] {error/ex2_fd_nb_300.dat};\addlegendentry{$\varepsilon_{a}$(FDBE)}
\addplot[JungleGreen,mark=*,mark size=1.2pt] table [x=NB,y=L2] {error/ex2_fd_nb_300.dat};;\addlegendentry{$\varepsilon_{0}$(FDBE)}
\nextgroupplot[
xlabel={\small{Number of oversampling layers}}, 
xtick={1,...,8},
ylabel={\small{Relative error}},
ymode=log,
]
\addplot[Blue,mark=triangle*,mark size=1.2pt] table [x=OS,y=H1] {error/ex2_svd0_ov_200.dat};
\addplot[BrickRed,mark=diamond*,mark size=1.2pt] table [x=OS,y=L2] {error/ex2_svd0_ov_200.dat};
\addplot[ForestGreen,mark=square*,mark size=1.2pt] table [x=OS,y=H1] {error/ex2_fd_ov_300.dat};
\addplot[JungleGreen,mark=*,mark size=1.2pt] table [x=OS,y=L2] {error/ex2_fd_ov_300.dat};
\end{groupplot}
\node at ($(group c2r1) + (-3.5cm,-4.0cm)$) {\ref{grouplegend}}; 
\end{tikzpicture}
\caption{\label{fig:example2nb} Relative error between the CEM-GMsFEM-EIRK1 solution and the reference solution with increasing the number of local basis functions (left) and the number of oversampling layers (right) for problem \eqref{eq:example2} at the final time $T=0.2$.}
\end{figure}
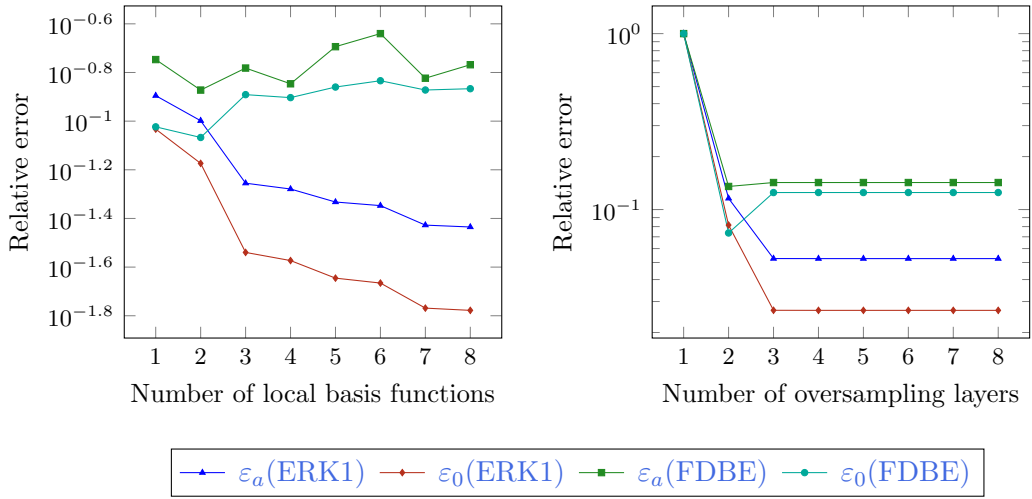
As in Example~\ref{example1}, Figure~\ref{fig:example2nb} (left) shows the relative errors with different numbers of local basis functions. Observe that errors decrease as the number of local basis functions increases using the CEM-GMsFEM-EIRK1 scheme in contrast to the CEM-GMsFEM-FDBE scheme. Therefore, we obtain a good agreement using only a few local basis functions on each coarse block. Furthermore, in Figure~\ref{fig:example2nb} (right), we show the error when we increase the numbers of oversampling layers using a fixed coarse grid size $H=\tfrac{1}{8}$ and $4$ basis functions on each coarse block. We observe that when the number of oversampling layers increases, the approximation becomes more accurate, and the error decreases very slowly once the number of oversampling layers attains a certain number. 

Table \ref{tab:example2} shows the $\H^{1}$, $\L^{2}$ and $\max$-error obtained at the final time simulation for the problem~\eqref{eq:example2}. We use $4$ basis functions on each coarse block with increasing the coarse grid size $H=\tfrac{1}{2},\tfrac{1}{4},\tfrac{1}{8}$,  and $\tfrac{1}{16}$, associated with some appropriate oversampling layers $m$. We obtain a good approximation for all coarse scale solutions with big-time steps by CEM-GMsFEM-EIRK1 compared with the CEM-GMsFEM-FDBE. In the latter, even when using a sufficiently small time step, the scheme does not converge and is more drastic than in example~\ref{example1}. In Table~\ref{tab:example22}, we show the temporal convergence order in both $\L^{2}$ and $\H^{1}$-norms. Similar to Example~\ref{example1}, we have that the order of $\H^{1}$-norm is about $1$ and  $\L^{2}$-norm is much better that theoretical result. Figure~\ref{fig:example2} depicts the solution profiles at the final instant $T=0.2$ using the two schemes and reference solution for the problem \eqref{eq:example2}. 
\begin{table}[!h]
\centering
\begin{tabular}{|c|c|c|c|c|c|}
\hline 
Scheme & $H$ & $m$ & $\varepsilon_{a}$ & $\varepsilon_{0}$ & $\varepsilon_{\infty}$\tabularnewline
\hline 
\hline 
\multirow{4}{*}{EIRK1} & $\frac{1}{2}$ & 1 & 1.3711E--01 & 9.1747E--02 & 1.2567E--01\tabularnewline
\cline{2-6} \cline{3-6} \cline{4-6} \cline{5-6} \cline{6-6} 
 & $\frac{1}{4}$ & 2 & 1.1267E--01 & 7.1108E--02 & 1.0623E--01 \tabularnewline
\cline{2-6} \cline{3-6} \cline{4-6} \cline{5-6} \cline{6-6} 
 & $\frac{1}{8}$ & 3 & 5.2644E--02 & 2.4200E--02 & 4.0774E--02
\tabularnewline
\cline{2-6} \cline{3-6} \cline{4-6} \cline{5-6} \cline{6-6} 
 & $\frac{1}{16}$ & 4 & 1.9081E--02 & 8.1719E--03 & 1.2419E--02
\tabularnewline
\hline 
\multirow{4}{*}{FDBE} & $\frac{1}{2}$ & 1 & 1.4664E--01 & 7.8413E--02 & 8.7485E--02\tabularnewline
\cline{2-6} \cline{3-6} \cline{4-6} \cline{5-6} \cline{6-6} 
 & $\frac{1}{4}$ & 2 & 1.3394E--01 & 8.3843E--02 & 9.9419E--02 \tabularnewline
\cline{2-6} \cline{3-6} \cline{4-6} \cline{5-6} \cline{6-6} 
 & $\frac{1}{8}$ & 3 & 1.4239E--01 & 1.2660E--01& 1.6325E--01\tabularnewline
 \cline{2-6} \cline{3-6} \cline{4-6} \cline{5-6} \cline{6-6} 
 & $\frac{1}{16}$ & 4 & 3.8840E--01 & 1.7666E--01 & 3.6423E--01\tabularnewline
\hline 
\end{tabular}
\caption{\label{tab:example2} Errors between the coarse-scale and the reference solution at the final time $T=0.2$ with different coarse grid sizes and different numbers of associated oversampling layers $m$ for problem~\eqref{eq:example2} using a contrast of $10^{3}$.}
\end{table}
\begin{table}
\centering
\begin{tabular}{|c|c|c|c|c|c|}
\hline 
Scheme & $\Nt$ & $\varepsilon_{a}$ & $CR$ & $\varepsilon_{0}$ & $CR$ \tabularnewline
\hline 
\hline 
\multirow{5}{*}{EIRK1} & 8 & 2.7427E+00 & -- & 2.7372E+00 & -- \tabularnewline
\cline{2-6} \cline{3-6} \cline{4-6} \cline{5-6} \cline{6-6}
 & 16 & 8.3854E--01 & 1.7096E+00 & 8.3422E--01 & 1.7142E+00\tabularnewline
\cline{2-6} \cline{3-6} \cline{4-6} \cline{5-6} \cline{6-6} 
 & 32 & 3.0339E--01 & 1.4667E+00 & 2.9734E--01 & 1.4882E+00\tabularnewline
\cline{2-6} \cline{3-6} \cline{4-6} \cline{5-6} \cline{6-6} 
 & 64 & 1.0627E--01 & 1.5134E+00 & 9.3312E--02 & 1.6720E+00\tabularnewline
 \cline{2-6} \cline{3-6} \cline{4-6} \cline{5-6} \cline{6-6} 
 & 128 & 4.7160E--02 & 1.1720E+00 & 6.3001E--03 & 3.8886E+00\tabularnewline
 \hline
\end{tabular}
\caption{\label{tab:example22} Temporal accuracy for problem~\eqref{eq:example2} with  contrast of $10^{3}$ and final time of $T=0.2$ with varying time steps $\Nt$ and fixed coarse size of $H=\tfrac{1}{8}$ and $4$ basis functions.}
\end{table}

\begin{figure}[!h]
\centering
\subfloat[EIexa2][EIRK1.]{\includegraphics[width=0.30\textwidth]{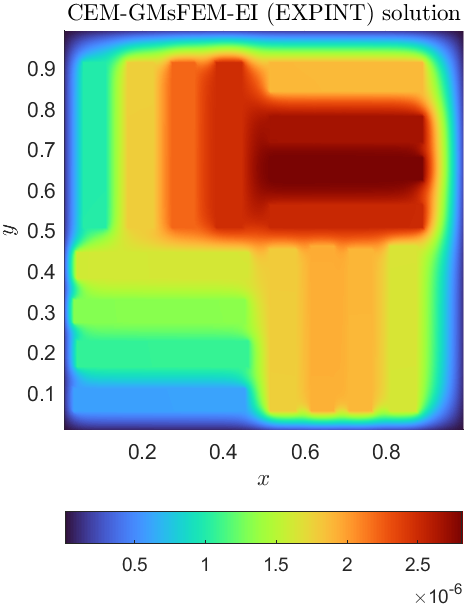}\label{fig:EIexa2}}
\subfloat[FDexa2][FDBE.]{\includegraphics[width=0.30\textwidth]{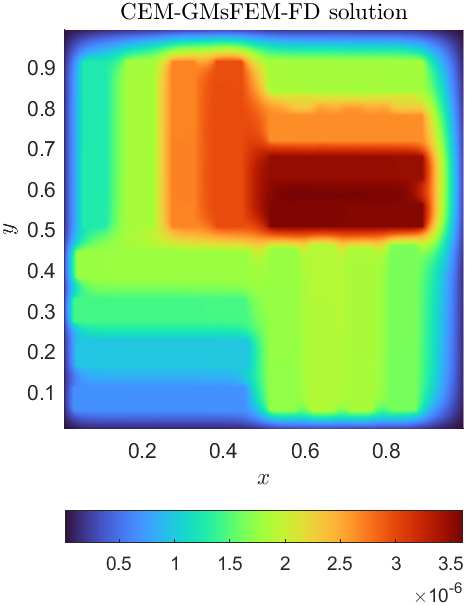}\label{fig:FDexa2}}
\subfloat[RFexa2][Reference solution.]{\includegraphics[width=0.30\textwidth]{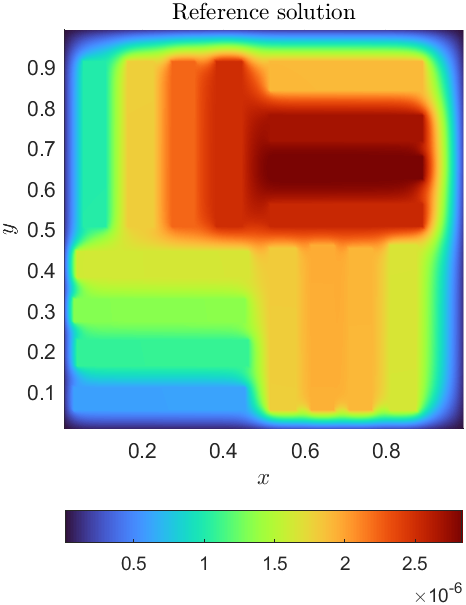}\label{fig:RFexa2}}
\caption{\label{fig:example2} (a) CEM-GMsFEM-EIRK1, (b) CEM-GMsFEM-FDBE, and (c) the reference solution of problem~\eqref{eq:example2} at final time $T=0.2$, using $H=\tfrac{1}{8}$, $4$ local multiscale basis functions, and $m=4$ oversampling layers.}
\end{figure}
\end{example}

In the examples below, we shall use a second-order explicit exponential Runge-Kutta and Crank-Nicolson schemes denoted by CEM-GMsFEM-EIRK22 and CEM-GMsFEM-FDCN, respectively.

\begin{example}
\label{example3}
We consider the problem:
\begin{equation}
\label{eq:example3}
\left\{\begin{split}
\pt u -\dive(\k_{3}(\x)\gd u)& = \frac{1}{\epsilon^{2}}(u-u^{3}),\quad\mbox{in }\Om\times[0,T],\\
u(x_{1},x_{2},t)& = 0,\quad\mbox{on }\pOm\times[0,T],\\
u(x_{1},x_{2},t=0)& = \epsilon x_{1}(1-x_{1})x_{2}(1-x_{2}),\quad\mbox{on }\Om.
\end{split}\right.
\end{equation}
Here, $\epsilon=0.1$ measures the interface thickness and we use the high-contrast media $\k_{3}$ with the value of $10^{4}$ in the high-contrast channels (see Figure~\ref{fig:perm3}). The final time of this simulation is $T=0.016$.  We consider $100$ and $500$ time steps for spatial accuracy for CEM-GMsFEM-ERK22 and CEM-GMsFEM-FDCN, respectively. We also consider uniform refined coarse grids with $H=\tfrac{1}{2},\tfrac{1}{4},\tfrac{1}{8}$ and $\tfrac{1}{16}$. The reference solution is approximated in the fine-scale grid using the Backward Euler scheme with $1000$ time steps. The errors for spatial accuracy in the $\L^{2}$ and $\H^{1}$-norms are reported in Table~\ref{tab:example31}. We notice that the spatial convergence rates for CEM-GMsFEM-EIRK22 are higher than one in both norms, as expected in contrast with CEM-GMsFEM-FDCN, which fails to maintain spatial accuracy in $\H^{1}$-norm. 
\begin{table}
\centering
\begin{tabular}{|c|c|c|c|c|c|c|}
\hline 
Scheme & $H$ & $m$ & $\varepsilon_{a}$ & $CR$ & $\varepsilon_{0}$ & $CR$ \tabularnewline
\hline 
\hline 
\multirow{4}{*}{EIRK22} & $\tfrac{1}{2}$ & 2 & 1.8813E--01 & -- & 4.4512E--02 & -- \tabularnewline
\cline{2-7} \cline{3-7} \cline{4-7} \cline{5-7} \cline{6-7} \cline{7-7}
 & $\frac{1}{4}$ & 3 & 1.3651E--01 & 4.6267E--01 & 3.0293E--02 & 5.5520E--01\tabularnewline
\cline{2-7} \cline{3-7} \cline{4-7} \cline{5-7} \cline{6-7} \cline{7-7}
 & $\frac{1}{8}$ & 4 & 4.5901E--02 & 1.5724E+00 & 1.0694E--02 & 1.5020E+00\tabularnewline
\cline{2-7} \cline{3-7} \cline{4-7} \cline{5-7} \cline{6-7} \cline{7-7}
 & $\frac{1}{16}$ & 5 & 1.7469E--02 & 1.3937E+00 & 4.1275E--03 & 1.3735E+00\tabularnewline
\hline 
\multirow{4}{*}{FDCN} & $\frac{1}{2}$ & 2 & 1.8865E--01 & -- & 4.1144E--02 & --\tabularnewline
\cline{2-7} \cline{3-7} \cline{4-7} \cline{5-7} \cline{6-7} \cline{7-7}
 & $\frac{1}{4}$ & 3 & 1.9529E--01 & 4.6429E--01 & 2.6865E--02 & 6.1495E--01\tabularnewline
\cline{2-7} \cline{3-7} \cline{4-7} \cline{5-7} \cline{6-7} \cline{7-7}
 & $\frac{1}{8}$ & 4 & 4.6756E--02 & 1.5481E+00 & 4.8988E--03 & 2.4552E+00\tabularnewline
\cline{2-7} \cline{3-7} \cline{4-7} \cline{5-7} \cline{6-7} \cline{7-7}
 & $\frac{1}{16}$ & 5 & 8.7217E-01 & -4.2213E+00 & 2.0979E-03 & 1.2234E+00\tabularnewline
\hline 
\end{tabular}
\caption{\label{tab:example31} Spatial convergence rate for problem~\eqref{eq:example3} at the final time $T=0.016$ with varying coarse grid size $H$ and oversampling coarse layers $m$ using a contrast of $10^{4}$.}
\end{table}

In Table~\ref{tab:example32}, we show that the temporal accuracy is about $1$. This is probably due to the influence of spatial accuracy, which coincides very well with the estimates given in Theorems~\ref{thm:RK22dt2H} and \ref{thm:RK22dt2H2}. Figure~\ref{fig:example3} depicts the solution profiles at the final instant $T=0.016$ using the two schemes and reference solution for the problem \eqref{eq:example3}.

\begin{table}
\centering
\begin{tabular}{|c|c|c|c|c|c|}
\hline 
Scheme & $\Nt$ & $\varepsilon_{a}$ & $CR$ & $\varepsilon_{0}$ & $CR$\tabularnewline
\hline 
\hline 
\multirow{5}{*}{EIRK22} & 8 & 1.4146E--01 & -- & 1.3598E--01 & -- \tabularnewline
\cline{2-6} \cline{3-6} \cline{4-6} \cline{5-6} \cline{6-6}
 & 16 & 8.0467E--02 & 8.1401E--01 & 6.8636E--02 & 9.8635E--01\tabularnewline
\cline{2-6} \cline{3-6} \cline{4-6} \cline{5-6} \cline{6-6} 
 & 32 & 5.5277E--02 & 5.4172E--01 & 3.4029E--02 & 1.0122E+00\tabularnewline
\cline{2-6} \cline{3-6} \cline{4-6} \cline{5-6} \cline{6-6} 
 & 64 & 2.7406E--02 & 1.0122E+00 & 1.6740E--02 & 1.0234E+00\tabularnewline
 \cline{2-6} \cline{3-6} \cline{4-6} \cline{5-6} \cline{6-6} 
 & 128 & 1.3411E--02 & 1.0310E+00 & 8.4554E--03 & 9.8542E--01\tabularnewline
\hline 
\multirow{5}{*}{FDCN} & 8 & 6.8599E+00 & -- & 2.0627E--01 & -- \tabularnewline
\cline{2-6} \cline{3-6} \cline{4-6} \cline{5-6} \cline{6-6}
 & 16 & 6.6305E+00 & 4.9063E--02 & 1.1583E--01 & 8.3254E--01\tabularnewline
\cline{2-6} \cline{3-6} \cline{4-6} \cline{5-6} \cline{6-6} 
 & 32 & 5.9835E+00 & 1.4813E--01 & 6.2416E--02 & 8.9204E--01\tabularnewline
\cline{2-6} \cline{3-6} \cline{4-6} \cline{5-6} \cline{6-6} 
 & 64 & 4.6921E+00 & 3.5076E--01 & 3.2232E--02 & 9.5340E--01\tabularnewline
 \cline{2-6} \cline{3-6} \cline{4-6} \cline{5-6} \cline{6-6} 
 & 128 & 2.6097E+00 & 8.4632E--01 & 1.5789E--02 & 1.0296E+00\tabularnewline
\hline 
\end{tabular}
\caption{\label{tab:example32} Temporal accuracy for problem~\eqref{eq:example3} with contrast of $10^{4}$ and final time of $T=0.016$ with varying time steps $\Nt$ and fixed coarse size of $H=\frac{1}{8}$ and $4$ basis functions.}
\end{table}

\begin{figure}[!h]
\centering
\subfloat[EIexa3][EIRK22.]{\includegraphics[width=0.30\textwidth]{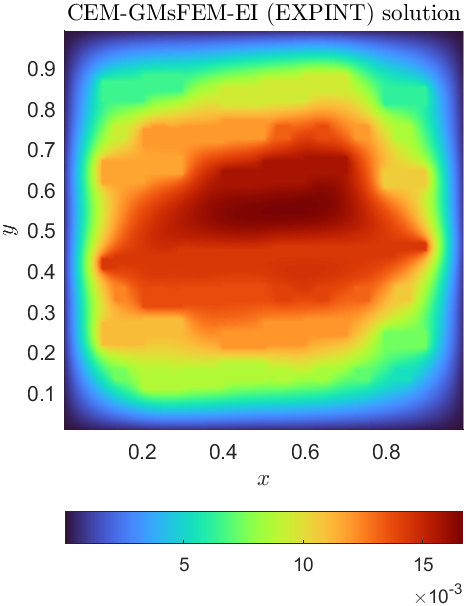}\label{fig:EIexa3}}
\subfloat[FDexa3][FDCN.]{\includegraphics[width=0.30\textwidth]{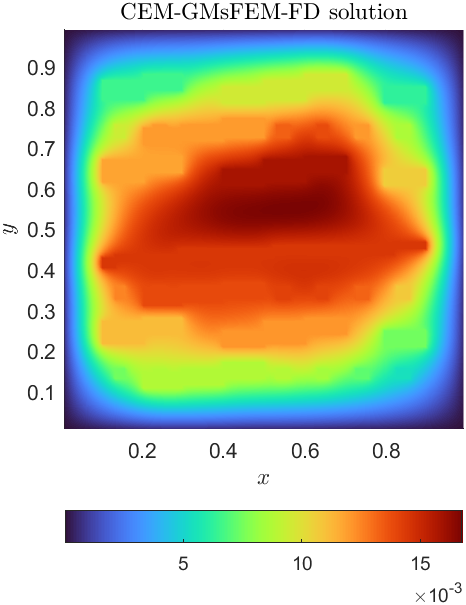}\label{fig:FDexa3}}
\subfloat[RFexa3][Reference solution.]{\includegraphics[width=0.30\textwidth]{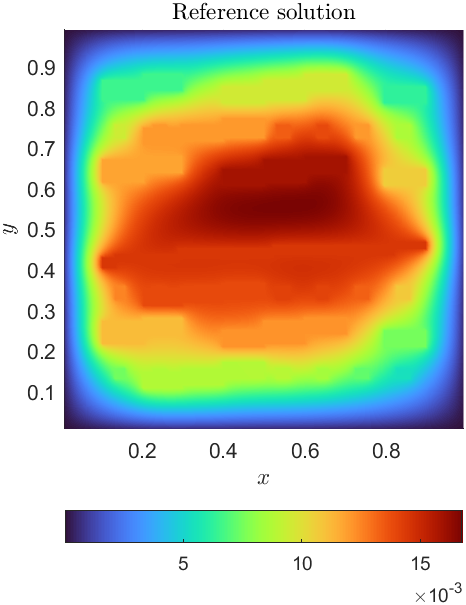}\label{fig:RFexa3}}
\caption{\label{fig:example3} (a) CEM-GMsFEM-EIRK22 (b) CEM-GMsFEM-FDCN, with $100$ and $500$ time-steps, respectively, and (c) the reference solution of problem~\eqref{eq:example3} at final time $T=0.016$, using coarse grid size $H=\tfrac{1}{8}$, and $4$ local multiscale basis functions, with $m=4$ oversampling layers.}
\end{figure}
\end{example}

\begin{example}
\label{example4}
We consider the problem:
\begin{equation}
\label{eq:example4}
\left\{\begin{split}
\pt u -\dive(\k_{4}(\x)\gd u)& = \frac{1}{\epsilon^{2}}(u-u^{3}),\quad\mbox{in }\Om\times[0,T],\\
u(x_{1},x_{2},t)& = 0,\quad\mbox{on }\pOm\times[0,T],\\
u(x_{1},x_{2},t=0)& = \tanh\left( \frac{0.25-\sqrt{(x_{1}-0.5)^{2}-(x_{2}-0.5)^{2}}}{\sqrt{\epsilon}}\right),\quad\mbox{on }\Om.
\end{split}\right.
\end{equation}
Here, $\epsilon=0.05$ measures the interface thickness and we use the high-contrast media $\k_{4}$ with the value of $10^{4}$ in the high-contrast channels (see Figure~\ref{fig:perm4}). Similarly to Example~\ref{example3}, we set the final time of $T=0.016$ and consider $100$ and $500$ time steps for CEM-GMsFEM-ERK22 and CEM-GMsFEM-FDCN, respectively. The reference solution is approximated in the fine-scale grid using the backward Euler scheme with $1000$ time steps. The errors for spatial accuracy in the $\L^{2}$ and $\H^{1}$-norms are reported in Table~\ref{tab:example41}. We notice that the spatial convergence rates for CEM-GMsFEM-EIRK22 and CEM-GMsFEM-FDCN are higher than one, but the first is slightly higher than the second. Table~\ref{tab:example42} reports the temporal convergence, which is about $1$ in both norms, but the CEM-GMsFEM-EIRK22 is slightly higher than the other one.  Figure~\ref{fig:example4} depicts the solution profiles at the final instant $T=0.016$ using the two schemes.
\begin{table}
\centering
\begin{tabular}{|c|c|c|c|c|c|c|}
\hline 
Scheme & $H$ & $m$ &$\varepsilon_{a}$ & $CR$ & $\varepsilon_{0}$ & $CR$\tabularnewline
\hline 
\hline 
\multirow{4}{*}{EIRK22} & $\tfrac{1}{2}$ & 2 & 1.8826E--01 & -- & 	3.9401E--02 & -- \tabularnewline
\cline{2-7} \cline{3-7} \cline{4-7} \cline{5-7} \cline{6-7} \cline{7-7}
 & $\frac{1}{4}$ & 3 & 1.6799E--01 & 1.6432E--01 & 	3.0689E--02 & 3.6048E--01\tabularnewline
\cline{2-7} \cline{3-7} \cline{4-7} \cline{5-7} \cline{6-7} \cline{7-7}
 & $\frac{1}{8}$ & 4 & 8.0971E--02 & 1.0529E+00 & 	1.0951E--02 & 	1.4866E+00\tabularnewline
\cline{2-7} \cline{3-7} \cline{4-7} \cline{5-7} \cline{6-7} \cline{7-7}
 & $\frac{1}{16}$ & 5 & 2.3621E--02 & 1.7773E+00 & 	4.3121E--03 & 	1.3446E+00\tabularnewline
\hline 
\multirow{4}{*}{FDCN} & $\frac{1}{2}$ & 2 & 1.8760E--01 & -- & 4.3753E--02 & --\tabularnewline
\cline{2-7} \cline{3-7} \cline{4-7} \cline{5-7} \cline{6-7} \cline{7-7}
 & $\frac{1}{4}$ & 3 & 1.6741E--01 & 1.6425E--01 & 	3.5695E--02 & 	2.9364E-01\tabularnewline
\cline{2-7} \cline{3-7} \cline{4-7} \cline{5-7} \cline{6-7} \cline{7-7}
 & $\frac{1}{8}$ & 4 & 8.1543E--02 & 1.0378E+00 & 1.8700E--02 & 	9.3262E-01\tabularnewline
\cline{2-7} \cline{3-7} \cline{4-7} \cline{5-7} \cline{6-7} \cline{7-7}
 & $\frac{1}{16}$ & 5 & 2.8306E--02 & 1.5264E+00 & 	1.4437E--02 & 	3.7328E--01\tabularnewline
\hline 
\end{tabular}
\caption{\label{tab:example41}  Spatial convergence rate for problem~\eqref{eq:example4} at the final time $T=0.016$ with varying coarse grid size $H$ and oversampling coarse layers $m$ using a contrast of $10^{4}$.}
\end{table}
 
\begin{table}
\centering
\begin{tabular}{|c|c|c|c|c|c|}
\hline 
Scheme & $\Nt$ & $\varepsilon_{a}$ & $CR$ & $\varepsilon_{0}$ & $CR$\tabularnewline
\hline 
\hline 
\multirow{5}{*}{EIRK22} & 8 & 5.8435E--01 & -- &	5.7640E--01 & -- \tabularnewline
\cline{2-6} \cline{3-6} \cline{4-6} \cline{5-6} \cline{6-6}
 & 16 & 3.0472E--01 & 9.3934E--01 & 	2.9074E--01 & 	9.8732E--01\tabularnewline
\cline{2-6} \cline{3-6} \cline{4-6} \cline{5-6} \cline{6-6} 
 & 32 & 1.5125E--01 & 1.0105E+00 & 	1.2894E--01 & 	1.1730E+00\tabularnewline
\cline{2-6} \cline{3-6} \cline{4-6} \cline{5-6} \cline{6-6} 
 & 64 & 9.4282E--02 & 6.8196E--01 & 	5.3729E--02 & 	1.2630E+00\tabularnewline
 \cline{2-6} \cline{3-6} \cline{4-6} \cline{5-6} \cline{6-6} 
 & 128 & 5.1733E--02 & 8.6590E--01 & 	2.0252E--02 & 1.4076E+00\tabularnewline
\hline 
\multirow{5}{*}{FDCN} & 8 & 8.5341E--01 & -- & 	8.4866E--01 & -- \tabularnewline
\cline{2-6} \cline{3-6} \cline{4-6} \cline{5-6} \cline{6-6}
 & 16 & 6.5339E--01 & 3.8530E--01	& 6.4422E--01 & 3.9763E--01\tabularnewline
\cline{2-6} \cline{3-6} \cline{4-6} \cline{5-6} \cline{6-6} 
 & 32 & 4.1568E--01 & 6.5245E--01 & 	4.0038E--01 & 6.8616E--01\tabularnewline
\cline{2-6} \cline{3-6} \cline{4-6} \cline{5-6} \cline{6-6} 
 & 64 & 2.3385E--01 & 8.2991E--01 & 	2.1352E--01 & 9.0701E--01\tabularnewline
 \cline{2-6} \cline{3-6} \cline{4-6} \cline{5-6} \cline{6-6} 
 & 128 & 1.3157E--01 & 8.2966E-01 & 1.0372E--01 & 1.0416E+00\tabularnewline
\hline 
\end{tabular}
\caption{\label{tab:example42} Temporal convergence rate for problem~\eqref{eq:example4} with contrast of $10^{4}$ and final time $T=0.016$ with varying time steps $\Nt$ and fixed coarse size of $H=\frac{1}{8}$ and $4$ basis functions.}
\end{table}

\begin{figure}[!h]
\centering
\subfloat[EIexa4][EIRK22.]{\includegraphics[width=0.30\textwidth]{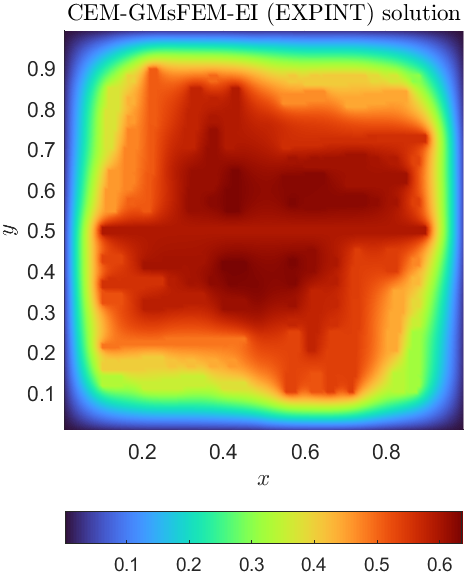}\label{fig:EIexa4}}
\subfloat[FDexa4][FDCN.]{\includegraphics[width=0.30\textwidth]{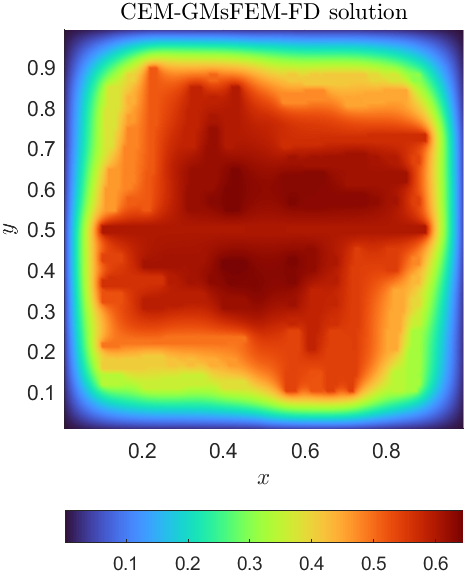}\label{fig:FDexa4}}
\subfloat[RFexa4][Reference solution.]{\includegraphics[width=0.30\textwidth]{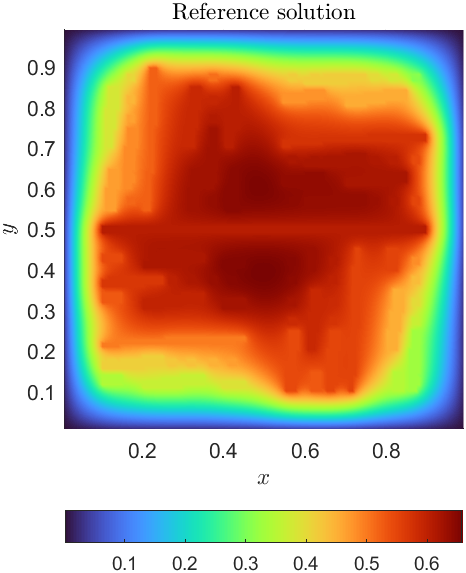}\label{fig:RFexa4}}
\caption{\label{fig:example4} (a) CEM-GMsFEM-EIRK22 (b) CEM-GMsFEM-FDCN, with $100$ and $500$ time-steps, respectively, and (c) the reference solution of problem~\eqref{eq:example4} at final time $T=0.016$, using coarse grid size $H=\frac{1}{8}$,  and $4$ local multiscale basis functions, with $m=4$ oversampling layers.}
\end{figure}

\end{example}


\begin{example}
\label{example5}
Finally, we consider the reaction-diffusion system:
\begin{equation}
\label{eq:example5}
\left\{\begin{split}
\pt u -\dive(\k_{3}(\x)\gd u)& = u-u^{3}-v,\quad\mbox{in }\Om_{1}\times[0,T],\\
\pt v -\dive(\k_{1}(\x)\gd v)& = u-v,\quad\mbox{in }\Om_{2}\times[0,T],\\
u(x_{1},x_{2},t)=v(x_{1},x_{2},t)& = 0,\quad\mbox{on }\pOm_{i}\times[0,T],\quad \mbox{for }i=1,2,\\
u(x_{1},x_{2},t=0)& = 0.05\sin(x_{1})\sin(x_{2}),\quad\mbox{on }\Om_{1},\\
v(x_{1},x_{2},t=0)& = \sin(\pi(x_{1}-0.25))\cos(2\pi(x_{2}-0.125)),\quad\mbox{on }\Om_{2}.
\end{split}\right.
\end{equation}
Here, we use the high-contrast media $\k_{3}$ in $\Om_{1}$ and $\k_{4}$ in $\Om_{2}$ with the value of $10^{4}$. The final time of this simulation is $T=0.016$, and we consider $100$ time steps for CEM-GMsFEM-ERK22. The reference solution is approximated in the fine-scale grid using the backward Euler scheme with $1000$ time steps. Table \ref{tab:example51} shows the convergence behavior concerning the coarse mesh size in $\H^{1}$ and $\L^{2}$ for the problem~\eqref{eq:example5} at the final time $T$. We only use $4$ basis functions on each coarse block with varying coarse grid sizes $H=\tfrac{1}{2},\tfrac{1}{4},\tfrac{1}{8}$,  and $\tfrac{1}{16}$, associated with some appropriate oversampling layers $m$. We observe that the spatial accuracy is higher than one for solutions $u$ and $v$, which matches our theoretical results. Table~\ref{tab:example52} shows the temporal convergence in both $\L^{2}$ and $\H^{1}$-norms.  Therefore, we have a good performance of the CEM-GMsFEM-EIRK22 scheme. Figure~\ref{fig:example5} depicts the solution profiles at the final instant $T=0.016$ using the two schemes.

\begin{table}
\centering
\begin{tabular}{|c|c|c|c|c|c|c|}
\hline 
Solution & $H$ & $m$ & $\varepsilon_{a}$ & $CR$ & $\varepsilon_{0}$ & $CR$\tabularnewline
\hline 
\hline 
\multirow{4}{*}{$u$} & $\tfrac{1}{2}$ & 2 & 2.0263E--01 & -- & 	3.4693E--02
 & -- \tabularnewline
\cline{2-7} \cline{3-7} \cline{4-7} \cline{5-7} \cline{6-7} \cline{7-7}
 & $\frac{1}{4}$ & 3 & 1.7844E--01 & 1.8337E--01 & 1.8615E--02 & 8.9816E--01\tabularnewline
\cline{2-7} \cline{3-7} \cline{4-7} \cline{5-7} \cline{6-7} \cline{7-7}
 & $\frac{1}{8}$ & 4 & 1.0640E--01 & 7.4599E--01 & 	6.4657E--03 & 	1.5256E+00\tabularnewline
\cline{2-7} \cline{3-7} \cline{4-7} \cline{5-7} \cline{6-7} \cline{7-7}
 & $\frac{1}{16}$ & 5 & 2.9204E--02 & 1.8652E+00 & 	1.8762E--03 & 	1.7850E+00\tabularnewline
\hline 
\multirow{4}{*}{$v$} & $\frac{1}{2}$ & 2 & 3.3754E--01 & -- & 1.3785E--01 & --\tabularnewline
\cline{2-7} \cline{3-7} \cline{4-7} \cline{5-7} \cline{6-7} \cline{7-7}
 & $\frac{1}{4}$ & 3 & 2.7226E--01 & 3.1005E--01 & 	9.7684E--02 & 	4.9688E--01\tabularnewline
\cline{2-7} \cline{3-7} \cline{4-7} \cline{5-7} \cline{6-7} \cline{7-7}
 & $\frac{1}{8}$ & 4 & 7.6753E--02 & 1.8267E+00 & 1.1111E--02 & 	3.1362E+00\tabularnewline
\cline{2-7} \cline{3-7} \cline{4-7} \cline{5-7} \cline{6-7} \cline{7-7}
 & $\frac{1}{16}$ & 5 & 2.0903E--02 & 1.8765E+00 & 3.3035E--03 & 	1.7499E+00\tabularnewline
\hline 
\end{tabular}
\caption{\label{tab:example51}  Spatial convergence rate for problem~\eqref{eq:example5} at the final time $T=0.016$ with varying coarse grid size $H$ and oversampling coarse layers $m$ using a contrast of $10^{4}$.}
\end{table}

\begin{table}
\centering
\begin{tabular}{|c|c|c|c|c|c|}
\hline 
Solution & $\Nt$ & $\varepsilon_{a}$ & $CR$ & $\varepsilon_{0}$ & $CR$\tabularnewline
\hline 
\hline 
\multirow{5}{*}{$u$} & 8 & 1.3509E--01 & -- &	6.7057E--02 & -- \tabularnewline
\cline{2-6} \cline{3-6} \cline{4-6} \cline{5-6} \cline{6-6}
 & 16 & 1.1382E--01 & 1.1869E+00 & 	3.4229E--02 & 	1.9590E+00\tabularnewline
\cline{2-6} \cline{3-6} \cline{4-6} \cline{5-6} \cline{6-6} 
 & 32 & 1.0823E--01 & 1.0517E+00 & 	1.8507E--02 & 	1.8496E+00\tabularnewline
\cline{2-6} \cline{3-6} \cline{4-6} \cline{5-6} \cline{6-6} 
 & 64 & 1.0682E--01 & 1.0132E+00 & 	1.1017E--02 & 	1.6799E+00\tabularnewline
 \cline{2-6} \cline{3-6} \cline{4-6} \cline{5-6} \cline{6-6} 
 & 128 & 1.0647E--01  & 1.0033E+00 & 7.5750E--03 & 	1.4544E+00\tabularnewline
\hline 
\multirow{5}{*}{$v$} & 8 & 2.0157E--01 & -- & 	1.4416E--01 & -- \tabularnewline
\cline{2-6} \cline{3-6} \cline{4-6} \cline{5-6} \cline{6-6}
 & 16 & 1.1423E--01 & 1.7646E+00 & 6.4472E--02 & 2.2360E+00\tabularnewline
\cline{2-6} \cline{3-6} \cline{4-6} \cline{5-6} \cline{6-6} 
 & 32 & 8.5618E--02 & 1.3342E+00 & 2.8341E--02 & 2.2749E+00\tabularnewline
\cline{2-6} \cline{3-6} \cline{4-6} \cline{5-6} \cline{6-6} 
 & 64 & 7.8317E--02 & 1.0932E+00 & 1.3250E--02 & 2.1390E+00\tabularnewline
 \cline{2-6} \cline{3-6} \cline{4-6} \cline{5-6} \cline{6-6} 
 & 128 & 7.6890E--02 & 1.0186E+00 & 1.0147E--02 & 1.3058E+00\tabularnewline
\hline 
\end{tabular}
\caption{\label{tab:example52} Temporal convergence rate for problem~\eqref{eq:example5} at the final time $T=0.016$ with varying coarse grid size $H$ and oversampling coarse layers $m$ using a contrast of $10^{4}$.}
\end{table}

\begin{figure}[!h]
\centering
\subfloat[EIexa5][ERK22.]{\includegraphics[width=0.50\textwidth]{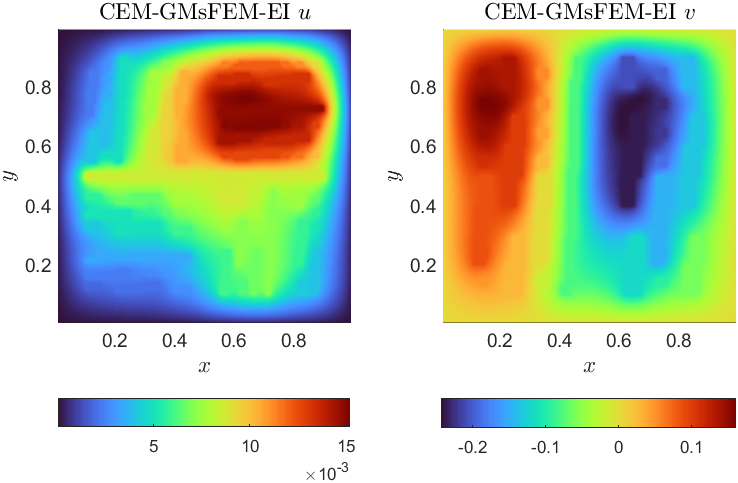}\label{fig:EIexa5}}
\subfloat[RFexa5][Reference solution.]{\includegraphics[width=0.50\textwidth]{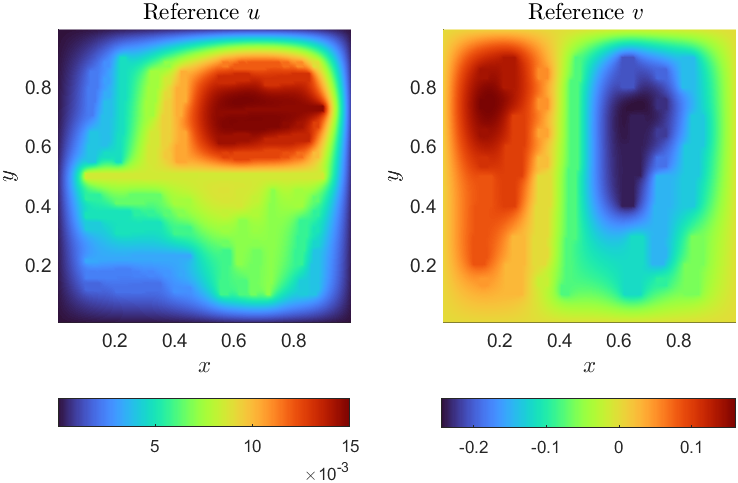}\label{fig:RFexa5}}
\caption{\label{fig:example5} (a) Second-order CEM-GMsFEM-ERK22 with $100$ time-steps and (b) The reference solution of problem~\eqref{eq:example5} at final time $T=0.016$, using $H=1/8$, $4$ local multiscale basis functions, and $m=4$ oversampling layers.}
\end{figure}
\end{example}

\section{Conclusion}
\label{sec:conclusion}
We have presented the explicit exponential integration CEM-GMsFEM for solving the semilinear parabolic problems in high-contrast media. As noted by \cite{contreras2023exponential}, the disparity of scales in the heterogeneous coefficients can affect the stability of the usual implicit schemes. In this work, we presented an alternative technique to handle this kind of scenario. We have used CEM-GMsFEM for spatial discretization. The first step is constructing the auxiliary space by solving local spectral problems. Next, we solve a constraint energy minimization problem to construct the multiscale basis functions in the oversampling regions. We introduce the first- and second-order explicit exponential Runge-Kutta integrators for temporal discretization. Rigorous convergence analysis of the proposed method shows optimal error estimates in the $\H^{1}$- and $\L^{2}$-norms with one and two Runge-Kutta stages, respectively. Extensive numerical examples all verify the spatial and temporal accuracy of the proposed scheme and confirm the theoretical results.

\section*{Acknowledgement}

The research of EC is partially supported by the Hong Kong RGC General Research Fund (Projects: 14305222 and 14304021). 

\appendix

\section{Convergence analysis}
\label{sec:convergence}
This section focuses on error estimates of fully discrete solutions produced by the proposed exponential integrator multiscale finite element method for solving the semilinear parabolic problem with homogeneous Dirichlet boundary conditions. We start with some notations and basic approximation results for the multiscale finite element approximations to estimate the error bound. We define the following norms for our analysis
\[
\|u\|_{a}^{2}:=\int_{\Om}\k|\gd u|^{2}d\x,\quad\|u\|_{s}^{2}:=\int_{\Om}\wk|u|^{2}d\x.
\]
In our estimates, we assume that the oversampling size is $m=\Or(\log(\k_{\max}/H))$, see \cite{chung2018constraint}. We introduce from \cite{thomee2006galerkin,huang2023efficient} some regularity and growth conditions for functions $f$ and $u$ to carry out the error analysis of our proposed method.

\begin{asp}
\label{asp:01}
The function $f(v)$ grows mildly with respect to $v$, i.e., there exists a number $p>0$ for $d=1,2$ or $p\in(0,2]$ for $d=3$ such that
\begin{equation}
\label{eq:asp1}
\left|\frac{\p f}{\p v}(v)\right|\preceq 1 + \|v\|^{p},\quad\feac v\in\R.
\end{equation}
\end{asp}
\begin{asp}
\label{asp:02}
The function $f(v(t))$ is smooth enough with respect to $t$,~i.e., for any given constant $C>0$, it holds
\begin{equation}
\label{eq:asp2}
\sum_{|\alpha|\leq2}\left|D^{\alpha}f(v(t))\right|\preceq 1,\quad\feac t\in[0,T],\quad v\in[-C,C].
\end{equation}
\end{asp}
\begin{asp}
\label{asp:03}
The exact solution $u(t)$ satisfies some of the following regularity conditions:
\begin{subequations}
\begin{align}
\sup_{0\leq t\leq T}\|u(t)\|_{2}&\preceq 1,\label{eq:asp3-1}\\
\sup_{0\leq t\leq T}\|\pt u(t)\|_{\infty}&\preceq 1,\label{eq:asp3-2}\\
\sup_{0\leq t\leq T}\|\ptt u(t)\|_{\infty}&\preceq 1\label{eq:asp3-3},
\end{align}
\end{subequations}
where the hidden constants may depend on $T$.
\end{asp}
We shall introduce the following result on the locally-Lipschitz continuity of the function $f$, see \cite{thomee2006galerkin}.

\begin{lem}
\label{lem:local-f}
Suppose that the function $f$ satisfies Assumption \ref{asp:01}, and the exact solution $u(t)$ satisfies \eqref{eq:asp3-1} in Assumption \ref{asp:03}. Then, $f$ is locally-Lipschitz continuous in a strip along the exact solution $u(t)$,~\ie, for any given positive constant $C$,
\begin{equation}
\label{eq:lem2-1}
\|f(v)-f(w)\|_{0}\preceq \|v-w\|_{a},
\end{equation}
for any $t\in[0,T]$ and $v,w\in\Vh$ satisfying $\max\{\|(v-u(t)\|_{a},\|w-u(t)\|_{a}\}\leq C$, where the hidden constant in \eqref{eq:lem2-1} may depend on $C$.
\end{lem}
\subsection{Fully-discrete error estimates}

This section presents the error between the exact solution $u(t_{n})$ and the fully discrete multiscale solution $\umsn$. For simplicity of presentation, we shall assume that the partition is uniform in $[0,T]$ with time step $\dt$. Let $\ums(t)$ be the multiscale solution of the semi-discrete problem \eqref{eq:sms-prob}, and $\umsn$ the multiscale fully-discrete solution produced by the exponential integrator multiscale finite element method \eqref{eq:ms-RK1} or \eqref{eq:ms-RK22}.


Let $\pu\in\Vms$ be the elliptic projection of the solution $u$ that satisfies
\begin{equation}
\label{eq:proju}
(\gd(u-\pu)(t),\gd v)=0,\quad\feac v\in\Vms,\mbox{and } t>0.
\end{equation}
The following lemma gives the error estimate of $\pu(t)$ for the semi-parabolic problem.

\begin{lem}
\label{lem:esti-proj}
Let $u$ be the solution of \eqref{eq:strong-prob}. For each $t\in[0,T]$, we define the elliptic projection $\pu\in\Vms$ satisfies \eqref{eq:proju}. Then,
\begin{subequations}
\begin{align}
\|(u-\pu)(t)\|_{a}&\preceq H\Lambda^{-\frac{1}{2}}\k_{\min}^{-\frac{1}{2}},\label{eq:lemE1-01}\\
\|(u-\pu)(t)\|_{0}&\preceq  H^{2}\Lambda^{-1}\k_{\min}^{-1}\label{eq:lemE1-02},
\end{align}
\end{subequations}
where $C$ is a constant independent of the mesh size $H$ and $\Lambda=\min_{1\leq i\leq N}\lambda_{L_{i}+1}^{(i)}$.
\end{lem}
\begin{proof}
Note that $u\in\Vz$ of \eqref{eq:strong-prob} satisfies
\[
a(u,v)=(f(u),v)-(\pt u,v)=(f(u)-\pt u,v),\quad\feac v\in\Vz,\quad\fall t\in[0,T].
\]
Thus, let $\pu(t)$ be the elliptic projection of $u$ in $\Vms$, that satisfies
\[
a(\pu,v)=a(u,v)=(f(u)-\pt u,v),\quad\feac v\in\Vms,\quad\fall t>0.
\]
For $v=u-\pu$ and by using Lemma 1 from \cite{chung2018constraint} we have,
\[
\|u-\pu\|_{a}^{2}=|a(u-\pu,u-\pu)|\leq \left|(f(u)-\pt u,u-\pu)\right|
 \leq \|\wk^{-1/2}(f(u)-\pt u)\|_{0}\|u-\pu\|_{s}.
\]
Observe that, by using the orthogonality of the eigenfunctions $\pij$ of \eqref{eq:eigen-prob}, we arrive at
\[
\|u-\pu\|_{s}^{2}=\sum_{i=1}^{\Nc}\|(I-\pi_{i})(u-\pu)\|_{s_{i}}^{2}\leq \Lambda^{-1}\sum_{i=1}^{\Nc}\|u-\pu\|_{a_{i}}^{2}=\Lambda^{-1}\|u-\pu\|_{a}^{2},
\]
where $\Lambda=\min_{1\leq i\leq N}\lambda_{L_{i}+1}^{(i)}$. Thus, by gathering the two expressions above, we have that
\[
\|u-\pu\|_{a}\leq \Lambda^{-\frac{1}{2}}\|\wk^{-\frac{1}{2}}(f(u)-\pt u)\|_{0}.
\]
By using $f(0)=0$, Assumption \ref{asp:01} and that $u$ satisfies the Assumption \ref{asp:03}, we obtain that $\|f(u)\|_{0}=\|f(u)-f(0)\|_{0}\preceq1$. Now, using \eqref{eq:asp3-2} and $|\gd\chi_{i}|=\Or(H^{-1})$, we can arrive at
\[
\|u-\pu\|_{a}\preceq H\Lambda^{-\frac{1}{2}}\k_{\min}^{-\frac{1}{2}},
\]
which is \eqref{eq:lemE1-01}. We shall use the duality argument for \eqref{eq:lemE1-02}. For $t\in[0,T]$, we define $w\in\Vz$ such that
\[
a(w,v)=(u-\pu,v),\quad\feac v\in\Vz,
\]
and define $\pw$ as the elliptic projection of $w$ in the space $\Vms$, that is,
\[
a(\pw,v)=(u-\pu,v),\quad\feac v\in\Vms.
\]
By using \eqref{eq:lemE1-01}, we obtain 
\[
\|u-\pu\|_{0}^{2}=a(w,u-\pu)=(w-\pw,u-\pu)\leq \|w-\pw\|_{a}\|u-\pu\|_{a}\preceq H^{2}\Lambda^{-1}\k_{\min}^{-1}\|u-\pu\|_{0}.
\]
Hence, dividing by $\|u-\pu\|_{0}$, we obtain \eqref{eq:lemE1-02}. This completes the proof.
\end{proof}

The following theorems give the $\H^{1}$ and $\L^{2}$ error estimates.
\begin{thm}
\label{thm:estiH1}
Let $u$ be the solution of \eqref{eq:weak-prob} and satisfies \eqref{eq:asp3-1} and \eqref{eq:asp3-2} in Assumption \ref{asp:03}. We define the elliptic projection $\pu\in\Vms$ satisfies \eqref{eq:proju}. There exists a constant $H_{0}>0$ such that if the coarse grid size $H\leq H_{0}$, then 
\begin{equation}
\label{eq:thmH1main}
\|(u-\ums)(\cdot,t)\|_{a}\preceq H\Lambda^{-\frac{1}{2}}\k_{\min}^{-\frac{1}{2}}+H^{2}\Lambda^{-1}\k_{\min}^{-1},
\end{equation}
where the hidden constant is independent of the coarse grid size $H$.
\end{thm}

\begin{proof}
We denote that 
\[
u-\ums=(u-\pu)+(\pu-\ums):=\rho +\theta,\quad \fall t\in[0,T],
\]
where $\pu$ is the elliptic projection in the space $\Vms$ of the exact solution $u$. 
About $\theta$, by \eqref{eq:sms-prob}, we obtain 
\begin{align*}
(\pt\theta,v)+a(\theta,v)&=(\pt \pu,v)-(f(\ums),v)+a(\pu,v)\\
& = (\pt (\pu-u),v)+(\pt u,v)+a(u,v)-(f(\ums),v)\\
& = (f(u)-f(\ums),v)-(\pt \rho,v),
\end{align*}
for all $v\in\Vms$. Taking $v=2\pt\theta=2\pt(\pu-\ums)\in\Vms$, we find that
\[
2(\pt\theta,\pt\theta)+2a(\theta,\pt\theta)=2(f(u)-f(\ums),\pt\theta)-2(\pt \rho,\pt\theta).
\]
By using the Cauchy-Schwarz and Young's inequality, we get that
\begin{align*}
2\|\pt\theta\|^{2}_{0}+\frac{d}{dt}\|\theta\|^{2}_{a}&\leq 2(\|f(u)-f(\ums)\|_{0}+\|\pt\rho\|_{0})\|\pt\theta\|_{0}\\
&\leq \|f(u)-f(\ums)\|^{2}_{0}+\|\pt\rho\|^{2}_{0}+2\|\pt\theta\|^{2}_{0}.
\end{align*}
Then,
\[
\frac{d}{dt}\|\theta\|^{2}_{a}\leq \|f(u)-f(\ums)\|^{2}_{0}+\|\pt\rho\|^{2}_{0}.
\]
Following \citet[Theorem 14.2, pp.~246]{thomee2006galerkin}, with $t$ large sufficiently in $[0,T]$, there exist $H_{0}>0$ such that the coarse grid size $H\leq H_{0}$, then, we can find that
\begin{align*}
\frac{d}{dt}\|\theta\|^{2}_{a}&\leq \|f(u)-f(\ums)\|^{2}_{0}+\|\pt\rho\|^{2}_{0}\\
& \leq \|f(u)-f(\pu)\|^{2}_{0}+\|f(\pu)-f(\ums)\|^{2}_{0}+\|\pt\rho\|^{2}_{0}\\
& \preceq \|\rho\|^{2}_{a}+\|\theta\|^{2}_{a}+\|\pt\rho\|^{2}_{0},
\end{align*}
where in the last inequality, we used Assumption \ref{asp:01} and fact that $\|\pu\|_{a}\leq C\|u\|_{a}\preceq 1$. Integrating with respect to time, we obtain
\[
\|\theta(\cdot,t)\|^{2}_{a}\preceq \int_{0}^{t}(\|\rho\|^{2}_{a}+\|\theta\|^{2}_{a})dt+\int_{0}^{t}\|\pt\rho\|^{2}_{0}ds+\|\theta(\cdot,0)\|^{2}_{a}.
\]
Note that the initial condition implies $\theta(0):=(\pu-\ums)(0)=0$. Then, Gronwall's inequality yields
\[
\|\theta(\cdot,t)\|^{2}_{a}\preceq \int_{0}^{t}(\|\rho\|^{2}_{a}+\|\pt\rho\|^{2}_{0})ds.
\]
By Lemma \ref{lem:esti-proj}, we obtain that
\begin{equation}
\label{eq:thetah1}
\|(\pu-\ums)(\cdot,t)\|^{2}_{a}\preceq\int_{0}^{t}\left(\|u-\pu\|^{2}_{0}+\|\pt(u-\pu)\|^{2}_{0}\right)ds\preceq H^{2}\Lambda^{-1}\k_{\min}^{-1}+H^{4}\Lambda^{-2}\k_{\min}^{-2}.
\end{equation}
Here, the hidden constant depends on $T$. Finally, using the triangle inequality and gathering \eqref{eq:lemE1-01} and \eqref{eq:thetah1}, we then have
\[
\|(u-\ums)(\cdot,t)\|^{2}_{a}\leq\|(u-\pu)(\cdot,t)\|^{2}_{a}+\|(\pu-\ums)(\cdot,t)\|^{2}_{a}\preceq H^{2}\Lambda^{-1}\k_{\min}^{-1}+H^{4}\Lambda^{-2}\k_{\min}^{-2}.
\]
This finishes the proof.
\end{proof}

\begin{thm}
\label{thm:estiL2}
Under assumptions of Theorem \ref{thm:estiH1}, we then have,
\[
\|(u-\ums)(\cdot,t)\|_{0}\preceq H^{2}\Lambda^{-1}\k_{\min}^{-1}.
\]
where $C$ is a constant independent of the mesh size $H$ and $\Lambda=\min_{1\leq i\leq N}\lambda_{L_{i}+1}^{(i)}$.
\end{thm}
\begin{proof}
Similar to Theorem \ref{thm:estiH1}, we shall denote
\[
u-\ums=(u-\pu)+(\pu-\ums):=\rho +\theta,\quad \fall t>0,
\]
where $\pu$ is the elliptic projection in the space $\Vms$ of the exact solution $u$. For $\theta$, taking $v=\theta=\pu-\ums\in\Vms$, we find that
\[
(\pt\theta,\theta)+a(\theta,\theta)=(f(u)-f(\ums),\theta)-(\pt \rho,\theta).
\]
By using $a(\theta,\theta)\geq 0$ and the Cauchy-Schwarz inequality, it follows that
\[
\frac{1}{2}\frac{d}{dt}\|\theta\|^{2}_{0}=\|\theta\|_{0}\frac{d}{dt}\|\theta\|_{0}\leq (\|f(u)-f(\ums)\|_{0}+\|\pt\rho\|_{0})\|\theta\|_{0}.
\]
By Lemma \ref{lem:local-f}, we obtain
\[
\frac{d}{dt}\|\theta\|_{0}\leq C\|u-\ums\|_{0}+\|\pt\rho\|_{0}\preceq(\|\rho\|_{0}+\|\theta\|_{0})+\|\pt\rho\|_{0}.
\]
Integrating in time and invoking Gronwall's inequality, we get
\[
\|\theta(\cdot,t)\|_{0}\preceq \int_{0}^{t}\|\rho\|_{0}ds +\int_{0}^{t}\|\pt\rho\|_{0}.
\]
By Assumption \ref{asp:01}, equations \eqref{eq:asp3-2},\eqref{eq:asp3-3} and  \eqref{eq:lemE1-02}, we obtain that
\begin{equation}
\label{eq:thetaL2}
\|(\pu-\ums)(\cdot,t)\|_{0}\preceq\int_{0}^{t}\left(\|u-\pu\|_{0}+\|\pt(u-\pu)\|_{0}\right)ds\preceq H^{2}\Lambda^{-1}\k_{\min}^{-1}.
\end{equation}
Finally, by using triangle inequality and gathering \eqref{eq:lemE1-02} and \eqref{eq:thetaL2},we obtain the bound
\[
\|(u-\ums)(\cdot,t)\|_{0}\leq\|(u-\pu)(\cdot,t)\|_{0}+\|(\pu-\ums)(\cdot,t)\|_{0}\preceq H^{2}\Lambda^{-1}\k_{\min}^{-1},
\]
which finishes the proof.
\end{proof}
%

We shall recall some definitions and results related to the semigroup $\{e^{-\dt\Lh}\}_{\dt\geq 0}$ and some terms used in the exponential Runge-Kutta schemes, which are important to the forthcoming analysis of the proposed multiscale finite element method. The following stability bounds for the semigroup $\{e^{-\dt\Lh}\}_{\dt\geq 0}$ are crucial in our analysis.

\begin{lem}[\citealp{hochbruck2005explicit}]
\label{lem:hochbruck2005explicit}
\,
\begin{enumerate}
\item For any given parameter $\gamma\geq 0$, it holds
\begin{equation}
\label{lem:eq01}
\|e^{-\dt\Lh}\|_{0}+\|(\dt)^{\gamma}\Lh^{\gamma}e^{-\dt\Lh}\|_{0}\preceq 1,\quad\feac\dt>0,\,\feac h>0.
\end{equation}
\item For any given parameter $0\leq\gamma\leq 1$, it holds
\begin{equation}
\label{lem:eq02}
\left\|(\dt)^{\gamma}\Lh^{\gamma}\sum_{j=1}^{n-1}e^{-j\dt\Lh}\right\|_{0}\preceq 1,\quad\feac\dt>0,\,\feac h>0.
\end{equation}
\item For any given parameter $0\leq\gamma\leq 1$, it holds
\begin{equation}
\label{lem:eq03}
\|\phi(-\dt\Lh)\|_{0}+\|(\dt)^{\gamma}\Lh^{\gamma}\phi(-\dt\Lh)\|_{0}\preceq 1,\quad\feac\dt>0,\,\feac h>0,
\end{equation}
where $\phi(-\dt\Lh)=\beta_{i}(-\dt\Lh)$ or $\phi(-\dt\Lh)=\alpha_{ij}(-\dt\Lh)$, $i,j=1,\dots,m$.
\end{enumerate}
\end{lem}
Observe that from coercivity of the bilinear form $a(\cdot,\cdot)$, we have that there exists a positive constant $C$ such that
\[
\frac{1}{C}\|\Lh^{\frac{1}{2}}v\|_{0}\leq\|v\|_{a}\leq C\|\Lh^{\frac{1}{2}}v\|_{0}.
\]
Similar to \eqref{eq:Lh-fem}, we have that the multiscale solution $\ums$ is defined as the solution of the following problem
\begin{equation}
\label{eq:Lms-cem}
\begin{cases}
\pt\ums(t)+\Lms\ums(t) &= \Pmsf(\ums),\quad\feac\x\in\Om,\quad 0\leq t\leq T,\\
\ums(0)& = \Pms\hat{u},\quad \x\in\Om,
\end{cases}
\end{equation}
where $\Pms$ is the $\L^{2}$-orthogonal projection operator in $\Vms$. We rewrite the semi-discrete solution $\ums(t_{n})$ ($n=1,\dots,\Nt$) into the sum of following expressions:
\begin{equation}
\label{eq:ms-exp}
\begin{split}
\ums(t_{n})&=e^{-\dt\Lms}\ums(t_{n-1})+\int_{0}^{\dt}e^{(s-\dt)\Lms}\Pmsf(\ums(s+t_{n-1}))ds\\
&=e^{-\dt\Lms}\ums(t_{n-1})+\int_{0}^{\dt}e^{(s-\dt)\Lms}\Pmsf(u(s+t_{n-1}))ds\\
& \quad+\int_{0}^{\dt}e^{(s-\dt)\Lms}\{\Pmsf(\ums(s+t_{n-1}))-\Pmsf(u(s+t_{n-1}))\}ds,
\end{split}
\end{equation}
and define the functions
\begin{equation}
\label{eq:xi-functions}
\left\{\begin{split}
\xi_{i}(-\dt\Lms)&=\phi_{i}(-\dt\Lms)-\sum_{k=1}^{m}\beta_{k}(-\dt\Lms)\frac{c_{k}^{i-1}}{(i-1)!},\quad i=1,\dots,m,\\
\xi_{j,i}(-\dt\Lms)&=\phi_{j}(-c_{i}\dt\Lms)c_{i}^{j}-\sum_{k=1}^{i-1}\alpha_{ik}(-\dt\Lms)\frac{c_{k}^{j-1}}{(j-1)!},\quad i,j=1,\dots,m.
\end{split}\right.
\end{equation}
We also denote $f^{(k)}(u(t))=\tfrac{d^{k}}{dt^{k}}f(u(t))$ as the $k$-th full differentiation of $f$ with respect to $t$. By comparing \eqref{eq:ms-exp} with the fully-discrete scheme \eqref{eq:RK-s}, we then obtain
\begin{align*}
\ums(t_{n-1}+c_{i}\dt) &= e^{-c_{i}\dt\Lms}\ums(t_{n-1})+\dt\sum_{j=1}^{i-1}\alpha_{ij}(-\dt\Lms)\Pmsf(u(t_{n-1}+c_{j}\dt))+\epsilon^{ni},\\
\ums(t_{n})&= e^{-\dt\Lms}\ums(t_{n-1})+\dt\sum_{i=1}^{m}\beta_{i}(-\dt\Lms)\Pmsf(u(t_{n-1}+c_{j}\dt))+\epsilon^{n},
\end{align*}
where the defects terms $\{\epsilon^{ni}\}_{i=1}^{m}$ and $\epsilon^{n}$ are respectively given by
\begin{align*}
\epsilon^{ni} &= \sum_{j=1}^{q}(\dt)^{j}\xi_{j,i}(-\dt\Lms)\Pmsf^{(j-1)}(u(t_{n-1}))+\epsilon^{ni,q},\\
\epsilon^{n}&= \sum_{i=1}^{q}(\dt)^{i}\xi_{i}(-\dt\Lms)\Pmsf^{(i-1)}(u(t_{n-1}))+\epsilon^{n,q},
\end{align*}
where the remainders $\epsilon^{ni,q}$ and $\epsilon^{n,q}$ are defined as
\begin{equation}
\label{eq:eniq}
\begin{split}
\epsilon^{ni,q} &= \int_{0}^{c_{i}\dt}e^{-(c_{i}\dt-s)\Lms}\int_{0}^{s}\frac{(s-\tau)^{q-1}}{(q-1)!}\Pmsf^{(q)}(u(t_{n-1}+\tau))d\tau ds\\
&\quad-\dt\sum_{k=1}^{i-1}\alpha_{ik}(-\dt\Lms)\int_{0}^{c_{k}\dt}\frac{(c_{k}\dt-\tau)^{q-1}}{(q-1)!}\Pmsf^{(q)}(u(t_{n-1}+\tau))d\tau\\
&\quad+\int_{0}^{c_{i}\dt}e^{-(c_{i}\dt-\tau)\Lms}\left\{\Pmsf(\ums(t_{n-1}+\tau))-\Pmsf(u(t_{n-1}+\tau))\right\}d\tau,
\end{split}
\end{equation}
and
\begin{align*}
\epsilon^{n,q} &= \int_{0}^{\dt}e^{-(\dt-s)\Lms}\int_{0}^{s}\frac{(s-\tau)^{q-1}}{(q-1)!}\Pmsf^{(q)}(u(t_{n-1}+\tau))d\tau ds\\
&\quad-\dt\sum_{i=1}^{m}\beta_{i}(-\dt\Lms)\int_{0}^{c_{i}\dt}\frac{(c_{i}\dt-\tau)^{q-1}}{(q-1)!}\Pmsf^{(q)}(u(t_{n-1}+\tau))d\tau\\
&\quad+\int_{0}^{\dt}e^{-(\dt-\tau)\Lms}\left\{\Pmsf(\ums(t_{n-1}+\tau))-\Pmsf(u(t_{n-1}+\tau))\right\}d\tau.
\end{align*}
In the expressions above, $q$ denotes any non-negative integers such that $\Pmsf^{(q)}(u(t))$ exists and is continuous. We have the next result following the framework given by \citet{hochbruck2005explicit,huang2023efficient}.

\begin{lem}
\label{lem:Hdt}
Given an integer $q=1$ or $2$. Suppose that the function $f$ satisfies the Assumptions \ref{asp:01} and \ref{asp:02}, and exact solution $u(t)$ fulfills \eqref{eq:asp3-1} and \eqref{eq:asp3-2} in Assumption \ref{asp:03}. In addition, if $q=2$, $u(t)$ satisfies \eqref{eq:asp3-3}. Then, for $n=1,\dots,\Nt$, $i=1,\dots,m$, it holds that
\begin{subequations}
\begin{align}
\|\epsilon^{ni,q}\|_{a}&\preceq (\dt)^{q+1}\sup_{0\leq\eta\leq 1}\|\Pmsf^{(q)}(u(t_{n-1}+\eta\dt))\|_{a}+H\Lambda^{-\frac{1}{2}}\k_{\min}^{-\frac{1}{2}},\label{eq1:lemHdt}\\
\left\|\sum_{j=0}^{n-1}e^{-j\dt\Lms}\epsilon^{n-j,q}\right\|_{a}&\preceq (\dt)^{q}\sup_{0\leq t\leq T}\|\Pmsf^{(q)}(u(t))\|_{a}+H\Lambda^{-\frac{1}{2}}\k_{\min}^{-\frac{1}{2}}\label{eq2:lemHdt},
\end{align}
\end{subequations}
where the hidden constants are independent of the coarse grid size $H$ and the time-step size $\dt$.
\end{lem}

\begin{proof}
We define $\|\epsilon^{ni,q}\|_{a}\leq I_{1}+I_{2}+I_{3}$, where $I_{k}$, $(k\in\{1,2,3\})$ denotes the energy norm of each term in the right-hand side of  \eqref{eq:eniq}. About \eqref{eq1:lemHdt}, by using condition \eqref{lem:eq01} in Lemma \ref{lem:hochbruck2005explicit}, we obtain
\begin{align*}
I_{1}&=\left\|\int_{0}^{c_{i}\dt}e^{-(c_{i}\dt-s)\Lms}\int_{0}^{s}\frac{(s-\tau)^{q-1}}{(q-1)!}\Pmsf^{(q)}(\ums(t_{n-1}+\tau))d\tau ds\right\|_{a}\\
&\preceq\left\|\int_{0}^{c_{i}\dt}e^{-(c_{i}\dt-s)\Lms}\int_{0}^{s}\frac{(s-\tau)^{q-1}}{(q-1)!}\Lms^{\frac{1}{2}}\Pmsf^{(q)}(\ums(t_{n-1}+\tau))d\tau ds\right\|_{0}\\
&\preceq(\dt)^{q+1}\sup_{0\leq s\leq c_{i}\dt}\|e^{-(c_{i}\dt-s)\Lms}\|_{0}\sup_{0\leq\eta\leq1}\|\Lms^{\frac{1}{2}}\Pmsf^{(q)}(\ums(t_{n-1}+\eta\dt))\|_{0}\\
&\preceq(\dt)^{q+1}\sup_{0\leq\eta\leq1}\|\Pmsf^{(q)}(\ums(t_{n-1}+\eta\dt))\|_{0}.
\end{align*}
Analogously, we have for $I_{2}$ that
\begin{align*}
I_{2}&= \left\|\dt\sum_{k=1}^{i-1}\alpha_{ik}(-\dt\Lms)\int_{0}^{c_{k}\dt}\frac{(c_{k}\dt-\tau)^{q-1}}{(q-1)!}\Pmsf^{(q)}(u(t_{n-1}+\tau))d\tau\right\|_{a}\\
&\preceq \left\|\dt\sum_{k=1}^{i-1}\alpha_{ik}(-\dt\Lms)\int_{0}^{c_{k}\dt}\frac{(c_{k}\dt-\tau)^{q-1}}{(q-1)!}\Lms^{\frac{1}{2}}\Pmsf^{(q)}(u(t_{n-1}+\tau))d\tau\right\|_{0}\\
&\preceq(\dt)^{q+1}\sum_{k=1}^{i-1}\left\|\alpha_{ik}(-\dt\Lms)\right\|_{0}
\sup_{0\leq\eta\leq1}\|\Pmsf^{(q)}(u(t_{n-1}+\eta\dt))\|_{a}\\
&\preceq(\dt)^{q+1}\sup_{0\leq\eta\leq1}\|\Pmsf^{(q)}(u(t_{n-1}+\eta\dt))\|_{a}.
\end{align*}
About $I_{3}$, we get that
\begin{align*}
\|I_{3}\|_{a}&=\left\|\int_{0}^{c_{i}\dt}e^{-(c_{i}\dt-s)\Lms}\left\{\Pmsf^{(q)}(\ums(t_{n-1}+\tau))-\Pmsf^{(q)}(\uh(t_{n-1}+\tau))\right\}d\tau\right\|_{a}\\
&\preceq\left\|\int_{0}^{c_{i}\dt}\Lms^{\frac{1}{2}}e^{-(c_{i}\dt-s)\Lms}ds\right\|_{0}\sup_{0\leq\eta\leq1}\|\Pmsf(\ums(t_{n-1}+\eta\dt))\\
&\quad-\Pmsf(\uh(t_{n-1}+\eta\dt))\|_{0}.
\end{align*}
For the first term on the right-hand side of the inequality above, we have
\[
\left\|\int_{0}^{c_{i}\dt}\Lms^{\frac{1}{2}}e^{-(c_{i}\dt-s)\Lms}ds\right\|_{0}=\max_{\lambda}\left|\int_{0}^{c_{i}\dt}\lambda^{\frac{1}{2}}e^{-(c_{i}\dt-s)\lambda}ds\right|\leq \max_{\lambda}|\lambda^{-\frac{1}{2}}|\preceq 1,
\]
where $\lambda$ is the eigenvalue of $a(\phi_{i,\ms},v) = \lambda (\phi_{i,\ms},v)$, for all $v\in\Vms$. Since $\Pmsf$ satisfies Lemma \ref{lem:local-f}, and invoking Theorem \ref{thm:estiH1}, we find that
\begin{equation}
\label{eq:I3lemHdt}
\begin{split}
\|I_{3}\|_{a}&= \left\|\int_{0}^{c_{i}\dt}e^{-(c_{i}\dt-s)\Lms}\left\{\Pmsf(\ums(t_{n-1}+\tau))-\Pmsf(\uh(t_{n-1}+\tau))\right\}d\tau\right\|_{a}\\
&\preceq \sup_{0\leq\eta\leq1}\|\Pmsf(\ums(t_{n-1}+\eta\dt))-\Pmsf(\uh(t_{n-1}+\eta\dt))\|_{0}\\
&\preceq\sup_{0\leq\eta\leq1}\|\ums(t_{n-1}+\eta\dt)-\uh(t_{n-1}+\eta\dt)\|_{a}\\
&\preceq H\Lambda^{-\frac{1}{2}}\k_{\min}^{-\frac{1}{2}}.
\end{split}
\end{equation}
Therefore, using the triangle inequality, we will obtain the desired result. Following a similar argument, we find the bounded for \eqref{eq2:lemHdt}. 
\end{proof}
For the sake of simplicity, we finally define $\en=\umsn-\ums(t_{n})$ and $E^{n-1,i}=U_{\ms}^{n-1,i}-\ums(t_{n-1}+c_{i}\dt)$ for $i=1,\dots,m$. Then, we can find that the recurrence relations
\begin{equation}
\label{eq:err-Eni}
\begin{split}
E^{n-1,i} & = e^{-c_{i}\dt\Lms}\enm+\dt\sum_{j=1}^{i-1}\alpha_{ij}(-\dt\Lms)\big[\Pmsf(U_{\ms}^{n-1,j})\\
&\quad-\Pmsf(\ums(t_{n-1}+c_{j}\dt))\big]-\epsilon^{ni},
\end{split}
\end{equation}
\begin{equation}
\label{eq:err-en}
\en = e^{-\dt\Lms}\enm+\dt\sum_{i=1}^{m}\beta_{i}(-\dt\Lms)\left[\Pmsf(U_{\ms}^{nj})-\Pmsf(\ums(t_{n-1}+c_{j}\dt))\right]-\epsilon^{n}.
\end{equation}
Now, we shall show the error estimates for the first-order Exponential Euler scheme.

\begin{thm}
\label{thm:RK1dtH}
Suppose that $f$ satisfies Assumptions \ref{asp:01} and \ref{asp:02}, and the exact solution $u(t)$ satisfies \eqref{eq:asp3-1} and \eqref{eq:asp3-2}. Then, There exists a constant $H_{0}>0$ such that if the spatial coarse grid size $H\leq H_{0}$, such that the multiscale solution $\umsn$ given by the Exponential Euler scheme \eqref{eq:ms-RK1} holds
\begin{equation}
\label{eq:RK1dtH}
\|u(t_{n})-\umsn\|_{a}\preceq \dt + H\Lambda^{-\frac{1}{2}}\k_{\min}^{-\frac{1}{2}},
\end{equation}
for $n= 1,\dots,\Nt$. The hidden constants are independent of the coarse grid size $H$ and time step size $\dt$.
\end{thm}

\begin{proof}
By the triangle inequality, we have
\[
\|u(t_{n})-\umsn\|_{a}\leq \|(u-\ums)(t_{n})\|_{a}+\|\ums(t_{n})-\umsn\|_{a}.
\]
Then, the first term on the left-hand side is bounded by Theorem \ref{thm:estiH1}. We shall concentrate on obtaining a bound for the second one. Then, we firstly set $m=1$ in the recurrence relation \eqref{eq:err-en}, getting 
\[
\ums(t_{n})-\umsn=:\en = e^{-\dt\Lms}\enm+\dt\phi_{1}(-\dt\Lms)\left[\Pmsf(\umsnm)-\Pmsf(u(t_{n-1}))\right]-\epsilon^{n},
\]
where
\[
\epsilon^{n}=\dt\phi_{1}(-\dt\Lms)\Pmsf(u(t_{n-1}))+\epsilon^{n,1},
\]
and
\begin{align*}
\epsilon^{n,1}& = \int_{0}^{\dt}e^{-(\dt-s)\Lms}\int_{0}^{s}\Pmsf'(u(t_{n-1}+\tau))d\tau ds\\
&\quad+\int_{0}^{\dt}e^{-(\dt-\tau)\Lms}\left\{\Pmsf(\ums(t_{n-1}+\tau))-\Pmsf(u(t_{n-1}+\tau))\right\}d\tau.
\end{align*}
Then, by recursive substitution, we find that
\begin{equation}
\label{eq:mainH1ERK1}
\begin{split}
\|\en\|_{a} &\leq \left\|\dt\sum_{j=1}^{n-1}e^{-(n-j-1)\dt\Lms}\phi_{1}(\dt\Lms)\left[\Pmsf(\ums^{j})-\Pmsf(u(t_{j}))\right]\right\|_{a}\\
&\quad+\left\|\sum_{j=0}^{n-1}e^{-j\dt\Lms}\epsilon^{n-j}\right\|_{a}\\
& = I_{1}+I_{2}.
\end{split}
\end{equation}
For $I_{1}$ in \eqref{eq:mainH1ERK1}, we use the triangle inequality to find that
\begin{align*}
I_{1}&=\left\|\dt\sum_{j=1}^{n-1}e^{-(n-j-1)\dt\Lms}\phi_{1}(-\dt\Lms)\left[\Pmsf(\ums^{j})-\Pmsf(u(t_{j}))\right]\right\|_{a}\\
&\leq \left\|\dt\phi_{1}(-\dt\Lms)\left[\Pmsf(\umsnm)-\Pmsf(u(t_{n-1}))\right]\right\|_{a}\\
&\quad+\left\|\dt\sum_{j=0}^{n-2}e^{-(n-j-1)\dt\Lms}\phi_{1}(-\dt\Lms)\left[\Pmsf(\ums^{j})-\Pmsf(u(t_{j}))\right]\right\|_{a}\\
&=I_{3}+I_{4}.
\end{align*}
About $I_{3}$, we have
\begin{align*}
I_{3}&\leq\left\|\Lms^{\frac{1}{2}}\int_{0}^{\dt}e^{-(\dt-s)\Lms}ds\right\|_{0}\|\Pmsf(\umsnm)-\Pmsf(u(t_{n-1}))\|_{0}\\
& \preceq\dt\sup_{0\leq s\leq \dt}\|\Lms^{\frac{1}{2}}e^{-(\dt-s)\Lms}\|_{0}\|\Pmsf(\umsnm)-\Pmsf(u(t_{n-1}))\|_{0}\\
&\preceq(\dt)^{\frac{1}{2}}\|\umsnm-u(t_{n-1})\|_{a}=(\dt)^{\frac{1}{2}}\|\varepsilon^{n-1}\|_{a}.
\end{align*}
On other hand, by using \eqref{lem:eq03} in Lemma \ref{lem:hochbruck2005explicit}, we obtain
\[
\|\phi_{1}(-\dt\Lms)\|_{0}=\left\|\frac{1}{\dt}\int_{0}^{\dt}e^{-(\dt-s)\Lms}ds\right\|_{0}\leq\sup_{0\leq s\leq \dt}\|e^{-(\dt-s)}\|_{0}\preceq 1.
\]
Then, for $I_{4}$, we find that
\begin{align*}
I_{4}&\preceq \left\|\dt\Lms^{\frac{1}{2}}\sum_{j=0}^{n-2}e^{-(n-1-j)\dt\Lms}\right\|_{0}\sup_{0\leq t\leq T}\|\Pmsf(\ums(t))-\Pmsf(u(t))\|_{0}\\
&\quad+\dt\sum_{j=0}^{n-2}\|\Lms^{\frac{1}{2}}e^{-(n-1-j)\dt\Lms}\|_{0}\|\Pmsf(\ums^{j})-\Pmsf(\ums(t_{j}))\|_{0}\\
&\preceq\sup_{0\leq t\leq T}\|\Pmsf(\ums(t))-\Pmsf(u(t))\|_{0}+\dt\sum_{j=0}^{n-2}t^{-\frac{1}{2}}_{n-j-1}\|\Pmsf(\ums^{j})-\Pmsf(\ums(t_{j}))\|_{0}\\
& \preceq H\Lambda^{-\frac{1}{2}}\k_{\min}^{-\frac{1}{2}}+\dt\sum_{j=0}^{n-2}t^{-\frac{1}{2}}_{n-j-1}\|\varepsilon^{j}\|_{0}.
\end{align*}
Therefore, we obtain
\begin{equation}
\label{eq:thmm1}
I_{1}\preceq (\dt)^{\frac{1}{2}}\|\varepsilon^{n-1}\|_{a} +H\Lambda^{-\frac{1}{2}}\k_{\min}^{-\frac{1}{2}}+\dt\sum_{j=0}^{n-2}t^{-\frac{1}{2}}_{n-j-1}\|\varepsilon^{j}\|_{a}.
\end{equation}
Finally, for $I_{2}$ in \eqref{eq:mainH1ERK1}, we have $m=1$ and using the definition \eqref{eq:xi-functions}, we can infer $\xi(-\dt\Lms)=0$, along with $\epsilon^{j}=\epsilon^{j,1}$, for $j=1,\dots,n$. Then by invoking \eqref{eq2:lemHdt} in Lemma \ref{lem:Hdt}, we have
\begin{equation}
\label{eq:thmm2}
I_{2}\preceq \dt \sup_{0\leq t\leq T}\|\Pmsf(u(t))\|_{a}+H\Lambda^{-\frac{1}{2}}\k_{\min}^{-\frac{1}{2}}.
\end{equation}
Gathering \eqref{eq:thmm1} and \eqref{eq:thmm2}, we have
\begin{align*}
\|\en\|_{a}&\preceq(\dt)^{\frac{1}{2}}\|\varepsilon^{n-1}\|_{a}+\dt\sum_{j=0}^{n-2}t^{-\frac{1}{2}}_{n-j-1}\|\varepsilon^{j}\|_{a}+\dt \sup_{0\leq t\leq T}\|\Pmsf(u(t))\|_{a}+H\Lambda^{-\frac{1}{2}}\k_{\min}^{-\frac{1}{2}}\\
&\preceq\dt\sum_{j=1}^{n-1}t^{-\frac{1}{2}}_{n-j}\|\varepsilon^{j}\|_{a}+\dt \sup_{0\leq t\leq T}\|\Pmsf(u(t))\|_{a}+H\Lambda^{-\frac{1}{2}}\k_{\min}^{-\frac{1}{2}}.
\end{align*}
By using a discrete version of Gronwall's inequality we get
\[
\|\en\|_{a}\preceq \dt+H\Lambda^{-\frac{1}{2}}\k_{\min}^{-\frac{1}{2}}.
\]
Therefore, gathering the inequality above with \eqref{eq:thmH1main}, we arrive at \eqref{eq:RK1dtH}, which finishes the proof.
\end{proof}

\begin{lem}
\label{lem:RK22Eji}
Let $f$ be a function that satisfies Assumptions \ref{asp:01} and \ref{asp:02}, and consider that the exact solution $u(t)$ of problem \eqref{eq:strong-prob} satisfies the conditions \eqref{eq:asp3-1} and \eqref{eq:asp3-2} in Assumption \ref{asp:03}. If $m\geq 2$, then it holds for any $0\leq n \leq \Nt$,
\begin{equation}
\label{eq:lems2}
\|E^{ni}\|_{a}\preceq \|\enm\|_{a}+(\dt)^{2}\sup_{0\leq\eta\leq 1}\|\Pmsf'(u(t_{n-1}+\eta\dt))\|_{a}+ H\Lambda^{-\frac{1}{2}}\k_{\min}^{-\frac{1}{2}}+\sup_{0\leq t\leq T}\|(\pu-\ums)(\cdot,t)\|_{a},
\end{equation}
where the hidden constant is independent of coarse grid size $H$ and time step size $\dt$.
\end{lem}

\begin{proof}
Following the definition of $E^{ni}$ in \eqref{eq:err-Eni}, we have
\begin{align*}
\|E^{ni}\|_{a}&\leq \left\|e^{-c_{i}\dt\Lms}\enm\right\|_{a}+\left\|\dt\sum_{j=1}^{i-1}\alpha_{ij}(-\dt\Lms)\Pmsf(U_{\ms}^{nj})-\Pmsf(\ums(t_{n-1}+c_{j}\dt)))\right\|_{a}\\
&\quad+\|\epsilon^{ni}\|_{a}\\
&= I_{1}+I_{2}+I_{3}.
\end{align*}
By using Lemma \ref{lem:hochbruck2005explicit}, we obtain for $I_{1}$,
\begin{equation}
I_{1}\preceq\|\enm\|_{a}.
\end{equation}
Using the similar arguments to obtain \eqref{eq:I3lemHdt} in Lemma~\ref{lem:Hdt}, we get for $I_{2}$
\begin{align*}
I_{2}&=\left\|\dt\sum_{j=1}^{i-1}\alpha_{ij}(-\dt\Lms)\left[\Pmsf(U^{nj})-\Pmsf(u(t_{n-1}+c_{j}\dt))\right]\right\|_{a}\\
&\preceq \sum_{j=1}^{i-1}(\dt)^{\frac{1}{2}}\|(\dt)^{\frac{1}{2}}\Lms^{\frac{1}{2}}\alpha_{ij}(-\dt\Lms)\|_{0}\max_{2\leq j\leq i-1}\|\Pmsf(U^{nj})-\Pmsf(u(t_{n-1}+c_{j}\dt))\|_{0}\\
&\preceq (\dt)^{\frac{1}{2}}\max_{2\leq j\leq i-1}\|\Pmsf(U^{nj})-\Pmsf(u(t_{n-1}+c_{j}\dt))\|_{0}\\
&\preceq (\dt)^{\frac{1}{2}}\max_{2\leq j\leq i-1}\|E^{nj}\|_{a} +H\Lambda^{-\frac{1}{2}}\k_{\min}^{-\frac{1}{2}}.
\end{align*}
Finally, for $I_{3}$, using the expression \eqref{eq:xi-functions} and the consistency conditions \eqref{eq:cons-cond}, we can infer that the function $\xi_{1,j}=0$ for $j=1,\dots,m$, and then the estimation of $\|\epsilon^{ni}\|_{a}$ can be obtained via $\|\epsilon^{ni,1}\|_{a}$. Then, we get that
\begin{align*}
\|E^{ni}\|_{a}&\preceq \|\enm\|_{a}+(\dt)^{2}\sup_{0\leq\eta\leq1}\|\Pmsf'(t_{n-1}+\eta\dt)\|_{a}+(\dt)^{\frac{1}{2}}\max_{2\leq j\leq i-1}\|E^{nj}\|_{a}+H\Lambda^{-\frac{1}{2}}\k_{\min}^{-\frac{1}{2}}.
\end{align*}
This completes the proof.
\end{proof}

\begin{thm}
\label{thm:RK22dt2H}
Suppose that $f$ satisfies Assumptions \ref{asp:01} and \ref{asp:02}, and the exact solution $u(t)$ satisfies \eqref{eq:asp3-1}--\eqref{eq:asp3-3}. Then, There exists a constant $H_{0}>0$ such that if the spatial coarse grid size $H\leq H_{0}$, such that the multiscale solution $\umsn$ given by the Exponential Euler scheme \eqref{eq:ms-RK22} holds
\begin{equation}
\|u(t_{n})-\umsn\|_{a}\preceq (\dt)^{2} + H\Lambda^{-\frac{1}{2}}\k_{\min}^{-\frac{1}{2}},
\end{equation}
for $n= 1,\dots,\Nt$. The hidden constants are independent of the coarse grid size $H$ and time step size $\dt$.
\end{thm}

\begin{proof}
Similar to Theorem~\ref{thm:RK1dtH}, we have 

\begin{equation}
\label{eq:mainRK22}
\|u(t_{n})-\umsn\|_{a}\leq \|(u-\ums)(t_{n})\|_{a}+\|\ums(t_{n})-\umsn\|_{a}.
\end{equation}
For the second term on the right-hand side, by definition \eqref{eq:xi-functions}, we obtain that $\xi_{1}(-\dt\Lms)=\xi_{2}(-\dt\Lms)=0$, then $\epsilon^{n}=\epsilon^{n,2}$, by using the Lagrange interpolation for $m=2$. By \eqref{lem:eq03} in Lemma \ref{lem:hochbruck2005explicit}, we have
\begin{equation}
\label{eq:RK22-1}
\begin{split}
&\left\|\dt\sum_{j=0}^{n-1}e^{-(n-1-j)\dt\Lms}\sum_{i=1}^{2}\beta_{i}(-\dt\Lms)\left[\Pmsf(U^{ji})-\Pmsf(u(t_{j}+c_{i}\dt))\right]\right\|_{a}\\
&\preceq \left\|\dt\sum_{i=1}^{2}\beta_{i}(-\dt\Lms)\left[\Pmsf(U^{n-1,i})-\Pmsf(u(t_{n-1}+c_{i}\dt))\right]\right\|_{a}\\
&\quad+\left\|\dt\sum_{j=0}^{n-2}e^{-(n-1-j)\dt\Lms}\sum_{i=1}^{2}\beta_{i}(-\dt\Lms)\left[\Pmsf(U^{ji})-\Pmsf(u(t_{j}+c_{i}\dt))\right]\right\|_{a}\\
&=:I_{1}+I_{2}.
\end{split}
\end{equation}
For $I_{1}$, we get that
\begin{equation}
\label{eq:RK22-2}
\begin{split}
I_{1}&\preceq\sum_{i=1}^{2}(\dt)^{\frac{1}{2}}\left\|(\dt)^{\frac{1}{2}}\Lms^{\frac{1}{2}}\beta_{i}(-\dt\Lms)\right\|_{0}\max_{1\leq i\leq 2}\left\|\Pmsf(U^{n-1,i})-\Pmsf(u(t_{n-1}+c_{i}\dt))\right\|_{0}\\
&\preceq(\dt)^{\frac{1}{2}}\max_{1\leq i\leq 2}\left\|\Pmsf(U^{n-1,i})-\Pmsf(u(t_{n-1}+c_{i}\dt))\right\|_{0}\\
&\preceq (\dt)^{\frac{1}{2}}\max_{1\leq i\leq 2}\|E^{n-1,i}\|_{a} +H\Lambda^{-\frac{1}{2}}\k_{\min}^{-\frac{1}{2}}+\sup_{0\leq t \leq T}\|(\pu-\ums)(\cdot,t)\|_{a},
\end{split}
\end{equation}
and
\begin{equation}
\label{eq:RK22-3}
\begin{split}
I_{2}&\preceq \left\|\dt\Lms^{\frac{1}{2}}\sum_{j=0}^{n-2}e^{-(n-1-j)\dt\Lms}\right\|_{0}\sup_{0\leq t \leq T}\left\|\Pmsf(u(t))-\Pmsf(\ums(t))\right\|_{0}\\
&\quad+\sum_{j=0}^{n-2}\dt\left\|\Lms^{\frac{1}{2}}e^{-(n-j-1)\dt\Lms}\right\|_{0}\max_{1\leq i\leq 2}\left\|\Pmsf(U^{ji})-\Pmsf(\ums(t_{j}+c_{i}\dt))\right\|_{0}\\
&\preceq \dt\sum_{j=0}^{n-2}t^{-\frac{1}{2}}_{n-j-1}\max_{1\leq i\leq 2}\|\Pmsf(U^{ji})-\Pmsf(\ums(t_{j}+c_{i}\dt))\|_{0}+H\Lambda^{-\frac{1}{2}}\k_{\min}^{-\frac{1}{2}}\\
&\quad+\sup_{0\leq t \leq T}\|(\pu-\ums)(\cdot,t)\|_{a}\\
&\preceq \dt\sum_{j=0}^{n-2}t^{-\frac{1}{2}}_{n-j-1}\max_{1\leq i\leq 2}\|E^{ji}\|_{a}+H\Lambda^{-\frac{1}{2}}\k_{\min}^{-\frac{1}{2}}+\sup_{0\leq t \leq T}\|(\pu-\ums)(\cdot,t)\|_{a}.
\end{split}
\end{equation}
By using \eqref{eq:RK22-1}--\eqref{eq:RK22-3}, and Lemma \ref{lem:Hdt}, we arrive at
\begin{equation}
\begin{split}
\|\en\|_{a}&\leq \left\|\dt\sum_{j=0}^{n-1}e^{-(n-1-j)\dt\Lms}\sum_{i=1}^{2}\beta_{i}(-\dt\Lms)\left[\Pmsf(U^{ji})-\Pmsf(u(t_{j}+c_{i}\dt))\right]\right\|_{a}\\
&\quad+\left\|\sum_{j=0}^{n-1}e^{-j\dt\Lms}\epsilon^{n-j,2}\right\|_{a}\\
&\preceq (\dt)^{\frac{1}{2}}\max_{1\leq i \leq 2}\|E^{n-1,i}\|_{a}+\dt\sum_{j=0}^{n-2}t^{\frac{1}{2}}_{n-j-1}\max_{i\leq i \leq 2}\|E^{ji}\|_{a}\\
&\quad+(\dt)^{2}\sup_{0\leq t\leq T}\|\Pmsf^{(2)}(u(t))\|_{a}+H\Lambda^{-\frac{1}{2}}\k_{\min}^{-\frac{1}{2}}.
\end{split}
\end{equation}
Using Lemma \ref{lem:RK22Eji}, and a discrete version of Gronwall's inequality, we find that
\[
\|\ums(t_{n})-\umsn\|_{1}\preceq (\dt)^{2}+H\Lambda^{-\frac{1}{2}}\k_{\min}^{-\frac{1}{2}}.
\]
We obtain the desired result by gathering the expression above with  \eqref{eq:mainRK22}.
\end{proof}

Similar computation allows us to obtain the following error estimates in $\L^{2}$-norm.

\begin{thm}[Error estimate in $\L^{2}$-norm]
\label{thm:RK22dt2H2}
Suppose that $f$ satisfies Assumptions \ref{asp:01} and \ref{asp:02}, and the exact solution $u(t)$ satisfies \eqref{eq:asp3-1}--\eqref{eq:asp3-3}. Then, There exists a constant $H_{0}>0$ such that if the spatial coarse grid size $H\leq H_{0}$, such that the multiscale solution $\umsn$ given by the Exponential Euler scheme \eqref{eq:ms-RK1} holds
\[
\|u(t_{n})-\umsn\|_{0}\preceq (\dt) + H^{2}\Lambda^{-1}\k_{\min}^{-1},
\]
and for \eqref{eq:ms-RK22},
\[
\|u(t_{n})-\umsn\|_{0}\preceq (\dt)^{2} + H^{2}\Lambda^{-1}\k_{\min}^{-1},
\]
for $n= 1,\dots,\Nt$. The hidden constants are independent of the coarse grid size $H$ and time step size $\dt$.
\end{thm}

\bibliographystyle{plainnat}
\bibliography{references_ei}
\end{document}